\documentclass[11pt,reqno,tbtags,a4paper]{amsart}
\usepackage{amssymb}
\usepackage{mathabx}
\usepackage{url}
\usepackage[square,numbers]{natbib}
\bibpunct[, ]{[}{]}{;}{n}{,}{,}

\title[The space $D$ in several variables]
{The space $D$ in several variables: random variables and higher moments}

\date{1 April, 2020}

\author{Svante Janson}
\thanks{Partly supported by the Knut and Alice Wallenberg Foundation}
\address{Department of Mathematics, Uppsala University, PO Box 480,
SE-751~06 Uppsala, Sweden}
\email{svante.janson@math.uu.se}
\urladdr{http://www.math.uu.se/svante-janson}


\numberwithin{equation}{section}

\renewcommand\le{\leqslant}
\renewcommand\ge{\geqslant}

\allowdisplaybreaks

\theoremstyle{plain}
\newtheorem{theorem}{Theorem}[section]
\newtheorem{lemma}[theorem]{Lemma}
\newtheorem{proposition}[theorem]{Proposition}
\newtheorem{corollary}[theorem]{Corollary}

\theoremstyle{definition}
\newtheorem{example}[theorem]{Example}
\newtheorem{definition}[theorem]{Definition}
\newtheorem{problem}[theorem]{Problem}
\newtheorem{remark}[theorem]{Remark}

\newtheorem*{ack}{Acknowledgement}

\theoremstyle{remark}

\newenvironment{romenumerate}[1][-10pt]{
\addtolength{\leftmargini}{#1}\begin{enumerate}
 \renewcommand{\labelenumi}{\textup{(\roman{enumi})}}%
 }{\end{enumerate}}

\newenvironment{alphenumerate}[1][-10pt]{
\addtolength{\leftmargini}{#1}\begin{enumerate}
 \renewcommand{\labelenumi}{\textup{(\alph{enumi})}}%
 }{\end{enumerate}}

\newcounter{oldenumi}
{\setcounter{oldenumi}{\value{enumi}}
\begin{romenumerate} \setcounter{enumi}{\value{oldenumi}}}
{\end{romenumerate}}

\newcounter{thmenumerate}
\newenvironment{thmenumerate}
{\setcounter{thmenumerate}{0}%
 \def\item{\par
 \refstepcounter{thmenumerate}\textup{(\roman{thmenumerate})\enspace}}
}
{}

\newcounter{xenumerate}   

\newcommand\pfitemx[1]{\par#1:}
\newcommand\pfitemref[1]{\pfitemx{\ref{#1}}}

\newcommand{\refT}[1]{Theorem~\ref{#1}}
\newcommand{\refC}[1]{Corollary~\ref{#1}}
\newcommand{\refL}[1]{Lemma~\ref{#1}}
\newcommand{\refR}[1]{Remark~\ref{#1}}
\newcommand{\refS}[1]{Section~\ref{#1}}
\newcommand{\refSS}[1]{Section~\ref{#1}}
\newcommand{\refP}[1]{Proposition~\ref{#1}}
\newcommand{\refD}[1]{Definition~\ref{#1}}
\newcommand{\refE}[1]{Example~\ref{#1}}

\newcommand{\refApp}[1]{Appendix~\ref{#1}}

\begingroup
  \divide\count255 by 60
  \multiply\count255 by -60
  \advance\count255 by \time
  \ifnum \count255 < 10 \xdef\klockan{\the\count1.0\the\count255}
  \else\xdef\klockan{\the\count1.\the\count255}\fi
\endgroup

\newcommand\nopf{\qed}   

\newcommand{\sumim}{\sum_{i=1}^m}

\newcommand{\prodil}{\prod_{i=1}^\ell}

\newcommand{\tensorim}{\bigotimes_{i=1}^m}

\newcommand\set[1]{\ensuremath{\{#1\}}}
\newcommand\bigset[1]{\ensuremath{\bigl\{#1\bigr\}}}

\newcommand\xpar[1]{(#1)}
\newcommand\bigpar[1]{\bigl(#1\bigr)}
\newcommand\Bigpar[1]{\Bigl(#1\Bigr)}
\newcommand\biggpar[1]{\biggl(#1\biggr)}
\newcommand\lrpar[1]{\left(#1\right)}

\newcommand\xcpar[1]{\{#1\}}

\newcommand\bigabs[1]{\bigl|#1\bigr|}

\newcommand\biggabs[1]{\biggl|#1\biggr|}
\newcommand\lrabs[1]{\left|#1\right|}
\def\rompar(#1){\textup(#1\textup)}    

\newcommand\innprod[1]{\langle#1\rangle}

\def\xexp(#1){e^{#1}}
\newcommand\ceil[1]{\lceil#1\rceil}
\newcommand\floor[1]{\lfloor#1\rfloor}

\newcommand\setm{\set{1,\dots,m}}

\newcommand\mm{[m]}

\newcommand\ntoo{\ensuremath{{n\to\infty}}}
\newcommand\Ntoo{\ensuremath{{N\to\infty}}}

\newcommand\norm[1]{\|#1\|}

\newcommand\downto{\searrow}
\newcommand\upto{\nearrow}
\newcommand\punkt{.\spacefactor=1000}    
    
\newcommand\ie{i.e\punkt}
\newcommand\eg{e.g\punkt}

\newcommand\cf{cf\punkt}
\newcommand{\as}{a.s\punkt}
\newcommand{\aex}{a.e\punkt}

\newcommand\bbR{\mathbb R}

\newcommand\bbQ{\mathbb Q}

\newcounter{CC}
\newcounter{cc}

\newcommand\E{\operatorname{\mathbb E{}}}
\renewcommand\P{\operatorname{\mathbb P{}}}

\newcommand\ga{\alpha}

\newcommand\gd{\delta}
\newcommand\gD{\Delta}

\newcommand\gam{\gamma}

\newcommand\gL{\Lambda}
\newcommand\go{\omega}
\newcommand\gO{\Omega}
\newcommand\gs{\sigma}

\newcommand\gY{\Upsilon}
\newcommand\eps{\varepsilon}

\newcommand\cB{\mathcal B}

\newcommand\cD{\mathcal D}

\newcommand\cF{\mathcal F}
\newcommand\cG{\mathcal G}

\newcommand\cL{{\mathcal L}}

\newcommand\cS{{\mathcal S}}

\newcommand\tA{\tilde A}

\newcommand\ett[1]{\boldsymbol1\xcpar{#1}}

\newcommand\etta{\boldsymbol1}

\newcommand\qw{^{-1}}

\newcommand\oi{\ensuremath{[0,1]}}
\newcommand\ooi{(0,1]}

\newcommand\setoi{\set{0,1}}

\newcommand\dd{\,\mathrm{d}}

\newcommand{\ui}{uniformly integrable}
\newcommand\rv{random variable}
\newcommand\lhs{left-hand side}
\newcommand\rhs{right-hand side}

\newcommand\doi{\ensuremath{D(\oi)}}
\newcommand\doix[1]{\ensuremath{D(\oi^{#1})}}

\newcommand\doim{\doix{m}}
\newcommand\doiml{\doix{m-\ell}}
\newcommand\coi{\ensuremath{C(\oi)}}
\newcommand\coix[1]{\ensuremath{C(\oi^{#1})}}

\newcommand\coim{\coix{m}}

\newcommand\tensor{\otimes}
\newcommand\bigtensor{\bigotimes}
\newcommand\ptensor{\widehat\tensor}
\newcommand\itensor{\widecheck\tensor}
\newcommand\tpx[1]{^{\tensor #1}}

\newcommand\tpl{\tpx{\ell}}

\newcommand\itpx[1]{^{\itensor #1}}
\newcommand\itpk{\itpx{k}}
\newcommand\itpl{\itpx{\ell}}
\newcommand\itpm{\itpx{m}}
\newcommand\ptpx[1]{^{\ptensor #1}}
\newcommand\ptpk{\ptpx{k}}
\newcommand\ptpl{\ptpx{\ell}}

\newcommand\xx{x^*}

\newcommand\q{^*}

\newcommand\tp{tensor product}

\newcommand\mxx[2]{#1_1 #2 \dots #2 #1_k}
\newcommand\mmx[2]{#1_1 #2 \dotsm #2 #1_m}
\newcommand\mmxx[2]{#1_1 #2 \dots #2 #1_m}
\newcommand\hI{\widehat I}
\newcommand\hU{\widehat U}

\newcommand\U{\mathrm U}
\newcommand\ps{probability space}

\newcommand\assep{\as{} separably valued}
\newcommand\wassep{weakly \as{} separably valued}

\newcommand\ofp{(\gO,\cF,\P)}

\newcommand\meas{measurable}
\newcommand\Bameas{Baire measurable}
\newcommand\Bormeas{Borel measurable}
\newcommand\wmeas{weakly measurable}

\newcommand\dmeas{$\cD$-measurable}

\newcommand\comp{^{\mathsf c}}

\newcommand\chI{C(\hI)}
\newcommand\chIm{C(\hI^m)}
\newcommand\chIml{C(\hI^{m-\ell})}
\newcommand\chIlm{C(\hI^{\ell m})}

\newcommand\gsf{$\gs$-field}
\newcommand\Ba{\mathsf{Ba}}
\newcommand\Id{\mathrm{I}}
\newcommand\MBa{M_{\Ba}}
\newcommand\mujt{\mu_{J;t_{j_1},\dots,t_{j_\ell}}}
\newcommand\bmu{\bar\mu}
\newcommand\tensorr{\tensor\dotsm\tensor}
\newcommand\itensorr{\,\itensor\dotsm\itensor\,}

\newcommand\itensorq{\,\itensor\,}
\renewcommand\emptyset{\varnothing}
\newcommand\xt{{\hat t}}

\newcommand\xu{{\hat u}}
\newcommand\vxt{\mathbf{\xt}}
\newcommand\vt{\mathbf{t}}
\newcommand\hA{\widehat A}
\newcommand\trho{\rho^*}
\newcommand\trhox{\rho^{*\sharp}}
\newcommand\tmu{\tilde{\mu}}
\newcommand\wset[1]{\widehat{\set{#1}}}
\newcommand\tgY{\widetilde\Upsilon}
\newcommand\oimi{\oi^{m-1}}
\newcommand\glf{\gL_f}
\newcommand\glif{\gL^i_f}
\newcommand\glxf{\gL^*_f}
\newcommand\glixf{\gL^{*i}_f}
\newcommand\gdxf{\gD^*_f}
\newcommand\gdxX{\gD^*_X}
\newcommand\xife{\Xi_{f,\eps}}
\newcommand\xiXe{\Xi_{X,\eps}}
\newcommand\xxi{\xi}
\newcommand\xxxi{\bar\xi}
\newcommand\ttx[1]{t'{}^{#1}}
\newcommand\kkg{k'_G}
\newcommand\gYa{\ga}
\newcommand\diff{\bigtriangleup}
\newcommand\NA{A}
\newcommand\tNA{\tA}

\newcommand\inthm{\int_{\hI^m}}

\newcommand\him{\hI^{m}}

\newcommand\hilm{\hI^{\ell m}}

\newcommand{\Holder}{H\"older}

\begin{document}

\begin{abstract} 
We study the Banach space $D([0,1]^m)$ of
functions of several variables that are (in a certain sense)
right-continuous with left limits, and extend several results previously
known for the standard case $m=1$.
We give, for example, 
a description of the dual space, and we show that a bounded
multilinear form always is measurable
with respect to the $\sigma$-field generated by the
point evaluations. 
These results are used to study random functions in the space.
(I.e., random elements of the space.)
In particular, we give results on existence of moments (in different
senses) of such random functions, and we 
give an application to the Zolotarev distance between two such random functions.
\end{abstract}

\maketitle

\section{Introduction}\label{S:intro}

Recall that $\doi$ is the set of real-valued functions on $I:=\oi$  that are
right-continuous with left limits, see \eg{} \cite[Chapter 3]{Billingsley}.
Similarly, the $m$-dimensional analogue $\doim$ is defined as the set of
real-valued functions $f$ on $\oi^m$ such that at every
$t=(t_1,\dots,t_m)\in\oi^m$, the limit of $f(s)$ exists 
 (as a finite real number),
as $s\to t$ in any
of the octants of the form $J_1\times\dots\times J_m$ where each $J_i$ is
either $[t_i,1]$ or $[0,t_i)$ (the latter only if $t_i>0$).
For example,
take $m=2$ for notational convenience; then
$f\in\doix2$ if and only if, for each $(x,y)\in\oi^2$, the limits
\begin{align*}
f(x+,y+)&:=\lim_{\substack{x'\to x,\; x'\ge x\\ y'\to y,\;y'\ge y}}f(x',y'), 
\\
f(x+,y-)&:=\lim_{\substack{x'\to x,\; x'\ge x\\ y'\to y,\;y'< y}}f(x',y'), 
\\
f(x-,y+)&:=\lim_{\substack{x'\to x,\; x'< x\\ y'\to y,\;y'\ge y}}f(x',y'), 
\\
f(x-,y-)&:=\lim_{\substack{x'\to x,\; x'< x\\ y'\to y,\;y'< y}}f(x',y')
\end{align*}
exist, except that we ignore all cases with an
argument $0-$. Note the slight asymmetry; we use $\ge$ but $<$. Note also
that necessarily $f(x+,y+)=f(x,y)$ when the limit exists.

The space $\doim$ was studied by 
\citet{Wichura-PhD,Wichura1969}
and 
\citet{Neuhaus};
the latter
extended the definition of the Skorohod topology from the case $m=1$
and proved many basic results on it. 
(The definition of the space in
\cite{Neuhaus} differs slightly from the one above
at the top and right parts of the boundary;
this is not essential and his proofs are just as valid for the version
considered here.) 
See also 
\citet{Straf} for an even more general setting.

In the present paper we study $\doim$ from a different point of view, as a
normed (Banach) space.
The space $\doi$ was studied as a normed space
in \cite[Chapter 9]{SJ271} 
(together with $\doim$ to some minor extent) in order to show some results on
second and higher moments of $\doi$-valued random variables;
these results were at least partly motivated by an application 
\cite{NeiningerSulzbach}
where convergence in distribution of some $\doi$-valued random variables
was shown by the contraction method, which required some of these results as
technical tools.
The purpose of the present paper is to extend some of these results 
for $\doi$ to $\doim$; one motivation is that 
this enables similar applications of the
contraction method to $\doim$-valued random variables,
see
\cite{BroutinSulzbach}.

Functions in $\doim$ are bounded, and we define 
\begin{align}\label{norm}
\norm f:=\sup_{t\in\oi^m}|f(t)|.  
\end{align}
$\doim$ is a Banach space with this norm. Note that the Banach space $\doim$
is not separable (already for $m=1$), and that the space $\coim$ of
continuous functions on $\oi^m$ is a closed, separable subspace.

\begin{remark}
We consider, for definiteness,  real-valued functions.
The definitions and results extend with no or trivial modifications to
complex-valued functions and measures. 
It is also easy to extend the results to vector-valued functions with values
in a fixed, finite-dimensional vector space.
\end{remark}

\begin{ack}
I am indebted to Henning Sulzbach for initiating this work by asking
me  questions that led to it, and also for helpful comments;
this almost led to a joint paper.
\end{ack}

\section{Preliminaries}\label{Sprel}

\subsection{The split interval}\label{SSsplit}

Define the \emph{split interval}
$\hI$ as the set consisting of two copies, $t$ and $t-$, 
of every point in $\ooi$,  together with a single $0$.
There is a natural total order on $\hI$, with $s<t-<t$ when $s,t\in\oi$
with $s<t$.
We define intervals in $\hI$ in the usual way, using
this order, and equip $\hI$ with the
order topology, which has a base consisting of all open intervals 
$[0,x)$, $(x,1]$, and $(x,y)$ with $x,y\in\hI$; then $\hI$ is a 
compact Hausdorff space;
see \eg{} \cite[Problems 1.7.4 and 3.12.3]{Engelking}.
The compact space $\hI$ is separable (\ie, has a countable dense subset, for
example the rational numbers) and 
first countable (every point has a countable neighbourhood basis),
but not second countable and not
metrizable, see
\eg{} \cite[Section 9.2]{SJ271}.

We regard $\oi$ as a subset of $\hI$, with the inclusion mapping
$\iota:\oi\to\hI$ given by $\iota(t)=t$. 
This mapping is \emph{not} continuous; 
the subspace topology on $I$ induced by $\hI$ is stronger than the usual
topology on $I$ (which we continue to use for $I$).
(The induced topology on $(0,1)$ yields a version of the \emph{Sorgenfrey line},
see \cite[Examples 1.2.2,  2.3.12, 3.8.14, 5.1.31]{Engelking}.)

Every function $f\in\doi$ has a (unique) extension to a continuous function
on $\hI$, given by $f(t-)=\lim_{s\upto t}f(s)$. Conversely, the restriction
to $I$ of any continuous function on $\hI$ is a function in $\doi$.
There is thus a bijection $\doi\cong C(\hI)$, which is an isometric
isomorphism as Banach spaces \cite{EdgarII}.
Another way to see this is to note that $\doi$ is a Banach algebra with
$\hI$ as its maximal ideal space, and that the Gelfand transform is this
isomorphism $\doi\cong C(\hI)$, see \cite[Section 9.2]{SJ271}.

These results extend immediately to $\doim$.
The definition of $\doim$ shows that
every function $f\in\doim$ has a (unique) extension to a continuous function
on $\hI^m$, and, conversely, that the restriction
to $I^m$ of any continuous function on $\hI^m$ is a function in $\doim$;
hence, there is a bijection $\doim\cong C(\hI^m)$, which is an isometric
isomorphism as Banach spaces.
Again, this can be regarded as
the Gelfand transform for the Banach algebra  $\doim$, with maximal ideal
space $\hI^m$.

\subsection{Tensor products}\label{SStensor}

For definitions and basic properties
of the injective and projective tensor products $X\itensor
Y$ and $X\ptensor Y$ of two Banach spaces $X$ and $Y$ see \eg{} 
\cite{Ryan}, or the summary in \cite{SJ271}.

In particular, recall that if $K$ is a compact Hausdorff space, then
$C(K)\itpk \cong C(K^k)$ (isometrically) by the natural identification of
$f_1\tensor\dotsm\tensor f_k$ with the function 
\begin{equation}
  \label{tensorf}
f_1\tensorr f_k(\mxx{x}{,}):=\prod_{i=1}^k f_i(x_i)
\end{equation}
on $K^k$.
In particular, linear combinations of such functions 
$\bigotimes_1^k f_i$ 
are dense in $C(K^k)$.
Furthermore, $C(K)$ has the approximation property (see \cite[Chapter
4]{SJ271}), and as a consequence, $C(K)\ptpk$ can be regarded as a
linear subspace of $C(K)\itpk=C(K^k)$ (with a different norm).

Since $\doim\cong C(\hI^m)$, these results apply also to $\doim$. In
particular,
$\doim\itpk\cong C(\hI^{mk})\cong\doix{mk}$, again
by the natural identification \eqref{tensorf} of
$\bigotimes_1^k f_i$ and $\prod_{i=1}^k f_i(x_i)$.
(From now on, we identify these spaces and write $=$ instead of $\cong$.)
In particular, 
linear combinations of functions 
$\bigotimes_1^m f_i=\prod_{i=1}^m f_i(x_i)$ with $f_i\in\doi$
are dense in $\doim$.
Furthermore,
$\doim$ has the approximation property and thus 
$\doim\ptpk\subset \doim\itpk=\doix{mk}$.
Note that $\doim\ptpk$ is \emph{not} a closed subspace of $\doim\itpk=\doix{mk}$
when $k\ge2$, and thus the projective and injective norms are not equivalent
on $\doim\ptpk$; see 
\eg{} \cite[Remark 7.9 and Theorem 9.27]{SJ271}.

\subsection{Baire sets and measures}\label{SSBaire}

If $K$ is a topological space,
then the \emph{Borel \gsf} $\cB(K)$ is the \gsf{} generated by the open sets in
$K$, and 
if $K$ is a compact Hausdorff space (the only case that we consider),
the \emph{Baire \gsf} $\Ba(K)$ is the \gsf{} generated by the
continuous real-valued functions on $K$; 
see \eg{} 
\cite[\S6.3]{Bogachev},
\cite[\S51]{Halmos} (with a somewhat different definition,
equivalent in the compact case)
and \cite[Exercises 7.2.8--13]{Cohn}.
Elements of $\cB(K)$ are called \emph{Borel sets}
and elements of $\Ba(K)$ are called \emph{Baire sets}.
A \emph{Baire (Borel) measure on $K$} is a measure on $\Ba(K)$
($\cB(K)$),
and similarly for signed measures;
we consider in this paper only finite measures.

We collect some basic properties.
\begin{lemma}\label{Lbaire}
Let $K,K_1,K_2$ be  compact Hausdorff spaces.
  \begin{romenumerate}[0pt]
  \item \label{baire-borel}
$\Ba(K)\subseteq\cB(K)$.
  \item \label{baire=}
If $K$ is a compact metric space, then $\Ba(K)=\cB(K)$.
\item \label{baire-times}
$\Ba(K_1\times K_2)=\Ba(K_1)\times \Ba(K_2)$.
\item \label{baire-map}
If\/ $(S,\cS)$ is any measurable space, then a function
$f:(S,\cS)\to(K,\Ba(K))$ is measurable if and only if $g\circ
f:(S,\cS)\to\allowbreak(\bbR,\cB)$ is measurable for every $g\in C(K)$.
\item \label{baire-cont}
If $f:K_1\to K_2$ is continuous, then
$f$ is Baire measurable, \ie,
$f:(K_1,\Ba(K_2))\to(K_2,\Ba(K_2))$ is measurable.
  \end{romenumerate}
\end{lemma}

\begin{proof}
  \ref{baire-borel}  and \ref{baire=} are easy and well-known;
for \ref{baire-times} see \cite[Theorem 51E]{Halmos}
or \cite[Lemma 6.4.2]{Bogachev}.
\ref{baire-map} is a consequence of the definition of $\Ba(K)$, and
\ref{baire-cont} follows.
\end{proof}

  By \refL{Lbaire}\ref{baire=}, there is no reason to study Baire sets
  instead of the perhaps more well-known Borel sets for
metrizable compact spaces, since they coincide.
However, we shall mainly study non-metrizable compact spaces such as $\hI$,
and then the 
Baire \gsf{} is often better behaved than the Borel \gsf.
One example is seen in \refL{Lbaire}\ref{baire-times} above; the
corresponding result for 
Borel \gsf{s} is not true in general, and in particular not for $K=\hI$,
see \refP{P2}.
Another important example is the Riesz representation theorem, which takes
the following simple form using Baire measures.
Let $\MBa(K)$ be the Banach space of signed Baire measures on $K$, with
$\norm\mu$ the total variation of $\mu$, 
\ie, $\norm\mu:=|\mu|(K)$, where the measure $|\mu|$ is the variation of $\mu$.

\begin{proposition}[The Riesz representation theorem]\label{Priesz}
Let $K$ be a compact Hausdorff space.
There is an isometric bijection between the space $C(K)^*$ of bounded
continuous linear functionals on $C(K)$ and the space $\MBa(K)$ of signed
Baire measures on $K$, where a signed Baire measure $\mu$ corresponds to the
linear functional $f\mapsto\int_Kf\dd\mu$.
\end{proposition}

\begin{remark}\label{Rriesz}
The Riesz representation theorem is perhaps more often stated in a
version using Borel measures, but then one has to restrict to \emph{regular}
signed Borel measures, see \eg{} \cite[Theorem 7.3.5]{Cohn}
or \cite[Theorem 7.10.4]{Bogachev}.
The connection between the two versions is that every
(signed) Baire measure on $K$ has a unique extension to a regular (signed)
Borel measure, see \cite[Theorem 54.D]{Halmos}
or \cite[Corollary 7.3.4]{Bogachev}.
\end{remark}

For a proof of \refP{Priesz}, see \cite[\S56]{Halmos}, or the references in
\refR{Rriesz} above.

\begin{example}\label{Ebaire}
  Although not needed for our results, it is interesting to note that
$\Ba(\hI)=\cB(\hI)$, but $\Ba(\hI^m)\subsetneq\cB(\hI^m)$, when $m\ge2$.
See \refApp{Abaire}.
\end{example}

\subsection{Some further notation}\label{SSnotation}
Let $\mm:=\setm$.

If $x\in\bbR$, then $\floor x$ and $\ceil x$ denote $x$ rounded down or up
to the nearest integer, respectively.

Recall that $t-$ is a point in $\hI\setminus I$ for $t\in\ooi$.
For completeness we define $0-:=0$.

Recall also that $\iota:I\to\hI$ denotes the inclusion mapping. 
Conversely, define the projection $\rho:\hI\to I$ by $\rho(t)=t$ and
$\rho(t-)=t$ for $t\in\oi$.
Let
\begin{equation}\label{phi}
  \phi:=\iota\circ\rho:\hI\to\hI
\end{equation}
be the composition of $\iota$ and $\rho$, \ie, the projection
\begin{equation}\label{phit}
  \begin{cases}
\phi(t)=t,&
\\
\phi(t-)=t.&
  \end{cases}
\end{equation}
Note that $\phi\circ\phi=\phi$, \ie, $\phi$ is a projection map.

If $A\subseteq\oi$, let $A-:=\set{t-:t\in A}\subset\hI$, and
$\hA:=\rho\qw(A)=A\cup{(A-)}\subset\hI$.
In particular, if $s\in\oi$, then $\widehat{\set{s}}=\set{s,{s-}}$.

We sometimes denote elements of $\hI^m$ by $\vxt=(\mmxx{\xt}{,})$.
Let $\pi_i:\hI^m\to\hI$  denote the projection 
on the $i$-th coordinate: $\pi_i(\vxt)=\xt_i$.

If $f\in\doi$ and $t\in\ooi$, let
\begin{equation}
  \gD f(t):=f(t)-f(t-).
\end{equation}
This defines a bounded linear map $\gD:\doi\to c_0(\ooi)$, with norm
$\norm{\gD}=2$;
see \cite[Theorem 9.1]{SJ271} for further properties.

We extend this to several dimensions by defining, for $f\in\doim$ and
$i\in\mm:=\setm$, 
\begin{equation}\label{gDi}
  \gD_i f(t_1,\dots,t_m):=f(t_1,\dots,t_m)-f(t_1,\dots,t_i-,\dots,t_m),
\end{equation}
\ie, the jump along the $i$-th coordinate at $t=(t_1,\dots,t_m)$.
(This is 0 when $t_i=0$, by our definition ${0-}=0$.)

We further define, for any subset $J=\set{j_1,\dots,j_\ell}\subseteq\mm$,
\begin{equation}\label{gDJ}
  \gD_J f := \gD_{j_1}\dotsm\gD_{j_\ell} f.
\end{equation}
Note that the operators $\gD_i$ commute, so their order in \eqref{gDJ}
does not matter.

\begin{remark}
  In particular, \eqref{gDi} shows that for $f_1,\dots,f_m\in\doi$,
\begin{equation}
  \gD_i(\mmx{f}{\tensor})=f_1\tensor\dots\tensor (\gD
  f_i)\tensor\dots\tensor f_m.
\end{equation}
Consequently, identifying $\doim$ and $\doi\itpm$ as in \refSS{SStensor},
\begin{equation}
\gD_i=\Id\itensorr\Id\itensorq \gD\itensorq\Id\itensorr \Id,
\end{equation}
where $\Id$ is the identity operator and there is a single $\gD$ in the
$i$-th position. Thus $\gD_i$ can be regarded as a bounded linear map into 
$\doi\itensorr \allowbreak c_0(\ooi)\itensorr\doi$, and similarly for $\gD_J$.
However, we will not use this point of view; we
just regard $\gD_if$ and $\gD_J f$ as the functions on $I^m$ given by 
\eqref{gDi}--\eqref{gDJ}.
\end{remark}

\section{Some projections}\label{Sproj}

Recall the mappings  $\iota$, $\rho$ and $\phi$ from \refSS{SSnotation}.

\begin{lemma}
\label{L1}  
\begin{thmenumerate}
\item \label{L1iota}
$\iota:I\to\hI$ is Baire measurable (but not continuous).
\item \label{L1rho}
$\rho:\hI\to I$ is continuous, and thus Baire measurable.
\item \label{L1phi}
$\phi:\hI\to\hI$ is Baire measurable (but not continuous).
\end{thmenumerate}
\end{lemma}
\begin{proof}
\pfitemref{L1iota}
  We have already remarked that $\iota$ is not continuous.

To see that $\iota$ is Baire measurable, 
\ie, that $\iota:(I,\Ba)\to(\hI,\Ba)$ is measurable,
let $g\in C(\hI)$.
Then $g\circ\iota\in\doi$, see \refSS{SSsplit}, and thus
$g\circ\iota: (I,\Ba)=(I,\cB)\to(\bbR,\cB)$ is measurable.
Thus $\iota$ is Baire measurable by \refL{Lbaire}\ref{baire-map}.

\pfitemref{L1rho}
The continuity follows from the definitions.

Alternatively, we may note that if $f\in C(I)$, then $f\in D(I)$, so it has
by \refSS{SSsplit} a continuous extension (which also is its Gelfand
transform) $\hat f\in C(\hI)$, given by $\hat f(t-)=f(t-)=f(t)$; hence 
$\hat f=f\circ\rho$.  In particular, taking $f$ to be the identity $i$ with
$i(x)=x$, we have $\hat i=\rho$, and thus $\rho\in C(\hI)$. 

\pfitemref{L1rho}
That $\phi=\iota\circ\rho$ is Baire measurable follows by \ref{L1iota} and
\ref{L1rho}. 
To see that $\phi$ is not continuous, it suffices to note that $\phi(\hI)=I$
is a proper dense subset of $\hI$, and thus not a compact subset of $\hI$.
\end{proof}

For a fixed $m$ and $1\le i\le m$, define $\phi_i=\phi_{i,m}:\hI^m\to\hI^m$ by 
\begin{equation}\label{phii}
  \phi_i(\xt_1,\dots,\xt_m):=(\xt_1,\dots,\phi(\xt_i),\dots,\xt_m)
\end{equation}
(with the identity in all coordinates except  the $i$-th).
Then $\phi_1,\dots,\phi_m$ are commuting projections $\hI^m\to\hI^m$.

Since $\phi$ is Baire measurable, and $\Ba(\hI^m)=\Ba(\hI)^m$ by
\refL{Lbaire}\ref{baire-times}, each $\phi_i$ is  Baire measurable.
Hence, $\phi_i$ induces a map
$\Phi_i:\MBa(\hI^m)\to \MBa(\hI^m)$ such that
\begin{equation}\label{b4}
  \int_{\hI^m} f\dd\Phi_i(\mu) = \int_{\hI^m}f\circ\phi_i\dd\mu,
\end{equation}
for Baire measurable and, say, bounded $f:\hI^m\to\bbR$.
Letting $\tau^\sharp$ denote the  map $\MBa\to\MBa$ induced 
as in \eqref{b4} by
a function $\tau:\hI^m\to\hI^m$, we thus have $\Phi_i=\phi_i^\sharp$,
and 
$\Phi_i\circ\Phi_i=\phi_i^\sharp\circ\phi_i^\sharp
=(\phi_i\circ\phi_i)^\sharp
=\phi_i^\sharp
=\Phi_i$.
Hence, $\Phi_i$ is a projection in $\MBa(\hI^m)$.
Similarly, $\Phi_i\circ\Phi_j=\Phi_j\circ\Phi_i$, so the projections
$\Phi_i$ commute (beacuse the projections $\phi_i$ do).

Let $\Psi_i:=\Id-\Phi_i$, where $\Id$ is the identity operator; thus, by
\eqref{b4},
\begin{equation}\label{b4Psi}
  \int_{\hI^m} f\dd\Psi_i(\mu) 
= \int_{\hI^m}\bigpar{f(\vxt)-f(\phi_i(\vxt))}\dd\mu(\vxt),
\end{equation}
for bounded Baire measurable $f$ on $\hI^m$.

Note that $\Psi_i$ also is a projection in $\MBa(\hI^m)$.
It follows immediately  that $\Phi_1,\dots,\Phi_m$ 
and $\Psi_1,\dots,\Psi_m$ are
commuting projections in $\MBa(\hI^m)$.
Furthermore,  for any $\mu\in\MBa(\hI^m)$,
\begin{align}
  \norm{\Phi_i(\mu)}&\le\norm{\mu},\label{normPhi}
\\
  \norm{\Psi_i(\mu)}&\le\norm{\mu}+  \norm{\Phi_i(\mu)}\le2\norm{\mu}.
\label{normPsi}
\end{align}

\begin{lemma}\label{LPsi}
If $\mu\in\MBa(\hI^m)$, then for each $i\in\setm$ there is a countable subset
$A_i\subset\ooi$ such that $\Psi_i (\mu)$ is supported on the set
$\pi_i\qw(\hA_i)=\set{(t_1,\dots,t_m)\in\hI^m:t_i\in \hA_i}$. 

Furthermore, if $s\in \ooi$, let $\Psi_i(\mu)_{s-}$ and $\Psi_i(\mu)_{s}$
denote the restrictions of $\Psi_i(\mu)$ to $\pi_i\qw(s-)$ and $\pi_i\qw(s)$,
respectively, regarded as measures on $\hI^{m-1}$; then
\begin{equation}
  \label{Psi+-}
\Psi_i(\mu)_{s-}=-\Psi_i(\mu)_{s}.
\end{equation}
\end{lemma}

\begin{proof}
For notational convenience, assume $i=1$,
and let $\bmu:=\Psi_1(\mu)$.

For each $s\in\ooi$, let 
$E_s:=\pi_1\qw(\widehat{\set{s}})=\set{{s-},s}\times\hI^{m-1}$,
and let
$\bmu_s$ denote the restriction of $\bmu$ to $E_s$.
($E_s$ is a Baire set in $\hI^m$, since \set{s} is measurable in $I$ and
$\pi_1$ and $\rho$ are Baire measurable as shown above; 
hence $\bmu_s$  is well defined.) 

For any finite set $F\subset\ooi $, 
\begin{equation}\label{treff}
\sum_{s\in F}\norm{\bmu_s}=\sum_{s\in F}|\bmu|(E_s)  
=|\bmu|\Bigpar{\bigcup_{s\in F}E_s}\le|\bmu|(\hI^m)
=\norm\bmu.
\end{equation}
Hence, 
$\sum_{s\in\ooi}\norm{\bmu_s}\le\norm\bmu$
and
$\bmu_s\neq0$ only for countably many $s$.
  
Let $A:=\set{s:\bmu_s\neq0}$ and $\nu:=\sum_{s\in\ooi}\bmu_s=\sum_{s\in A}\bmu_s$,
where the sum converges in $\MBa(\hI^m)$ by \eqref{treff}.
We shall prove that $\bmu=\nu$.
In order to show this, recall that $C(\hI^m)=C(\hI)\itensorq C(\hI^{m-1})$,
see \refSS{SStensor}, 
and thus linear combinations of
functions $f$ of the form
\begin{equation}\label{ff}
f(\xt_1,\dots,\xt_m)=f_1(\xt_1)g(\xt_2,\dots,\xt_m),  
\qquad  f_1\in C(\hI),\, g\in C(\hI^{m-1}),
\end{equation}
are dense in $C(\hI^m)$.
Hence,
it suffices
to show that
\begin{equation}\label{q1}
  \int_{\hI^m} f\dd\bmu=\int_{\hI^m} f\dd\nu
\end{equation}
for every $f$ as in \eqref{ff}.
For such $f$, \eqref{phii} and \eqref{phit} yield, for
$\vxt=(\xt_1,\dots,\xt_m)\in\hI^m$, 
\begin{equation}\label{q2}
  \begin{split} 
f(\vxt)- f(\phi_1(\vxt))
&=\bigpar{f_1(\xt_1)-f_1(\phi(\xt_1))}g(\xt_2,\dots,\xt_m)
\\&
=
\begin{cases}
0, & \xt_1=s\in I,  
\\
-\gD f_1(s)\, g(\xt_2,\dots,\xt_m), & \xt_1={s-}.
\end{cases}
  \end{split}
\end{equation}
Recall that $f_1\in C(\hI)=\doi$; regard $f_1$ as an element of $\doi$ and
let $D_{f_1}\subset\ooi$ be the countable set of discontinuities of $f_1$.
Then, by \eqref{q2}, 
$f(\vxt)- f(\phi_1(\vxt))=0$ unless $\xt_1=u-$ for some $u\in D_{f_1}$, and then
$\vxt\in E_u$. Consequently, \eqref{b4Psi} yields
\begin{equation}\label{b4Psi+}
  \int_{\hI^m} f\dd\bmu
=  \int_{\hI^m} f\dd\Psi_1(\mu) 
=\sum_{u\in D_{f_1}} \int_{E_u}\bigpar{f(\vxt)-f(\phi_1(\vxt))}\dd\mu(\vxt).
\end{equation}
Furthermore, applying \eqref{b4Psi} to the function $f\etta_{E_s}$, we also
find, for any $s\in\ooi$,
\begin{equation}\label{b4Psi++}
  \int_{E_s} f\dd\bmu
=  \int_{E_s} f\dd\Psi_1(\mu) 
= \int_{E_s}\bigpar{f(\vxt)-f(\phi_1(\vxt))}\dd\mu(\vxt).
\end{equation}
This integral vanishes when $s\notin A$, because then $\bmu=0$ on $E_s$, and
also when $s\notin D_{f_1}$, because then $f(\vxt)-f(\phi_1(\vxt))=0$ on $E_s$ by
\eqref{q2}. 
Consequently, \eqref{b4Psi+}--\eqref{b4Psi++} yield
\begin{equation}\label{b4Psi3}
  \int_{\hI^m} f\dd\bmu
=\sum_{u\in D_{f_1}} \int_{E_u}f\dd\bmu
=\sum_{u\in A} \int_{E_u}f\dd\bmu
=\int_{\hI^m}f\dd\nu,
\end{equation}
which verifies \eqref{q1} and thus $\bmu=\nu=\sum_{s\in A}\bmu_s$, which is
supported on $\bigcup_{s\in A}E_s=\pi_1\qw(\hA)$.

Finally, let $s\in \ooi$. Again, let $f$ be as in \eqref{ff}. Then, by
\eqref{b4Psi++} and \eqref{q2},
letting $\mu_{s-}$ denote the restriction of $\mu$ to $\set{s-}\times
\hI^{m-1}$, regarded as a measure on $\hI^{m-1}$,
and $\vxt':=(\xt_2,\dots,\xt_m)$,
\begin{equation}
  \begin{split}
\int_{E_s}f\dd\bmu 
&= \int_{\set{s-}\times\hI^{m-1}} (-\gD f_1(s) g(\vxt'))\dd\mu     
\\
&=\bigpar{f_1(s-)-f_1(s)}\int_{\hI^{m-1}} g(\vxt')\dd\mu_{s-}(\vxt')
\\&
=\int_{\hI^{m-1}} f(s-,\vxt')\dd\mu_{s-}(\vxt')
-\int_{\hI^{m-1}} f(s,\vxt')\dd\mu_{s-}(\vxt').
  \end{split}
\end{equation}
Hence, $\Psi_1(\mu)_{s-}=\mu_{s-}=-\Psi_1(\mu)_{s}$.
\end{proof}

\section{The dual space}

The continuous linear functionals on $\doi$ were described by
\citet{Pestman}, see also \cite[\S9.1]{SJ271}.
We extend this result to several dimensions as follows.

\begin{theorem}\label{Tdual}
  Every continuous 
linear functional $\chi$ on $\doim$ has a unique
  representation
  \begin{equation}
    \label{d1a}
\chi(f)=\sum_{J\subseteq\mm}\chi_J(f)
  \end{equation}
such that for every $J=\set{j_1,\dots,j_\ell}$, with $0\le\ell\le m$,
writing $J\comp:=\mm\setminus J=\set{j'_1,\dots,j'_{m-\ell}}$,
\begin{equation}\label{d1b}
  \chi_J(f)=\sum_{t_{j_1},\dots,t_{j_\ell}\in\ooi}
\int_{I^{m-\ell}} \gD_J f(t_1,\dots,t_m)
  \dd\mu_{J;t_{j_1},\dots,t_{j_\ell}}(t_{j'_1},\dots,t_{j'_{m-\ell}})
\end{equation}
where each $\mu_{J;t_{j_1},\dots,t_{j_\ell}}$ is a signed Borel measure on
$I^{m-\ell}$ and
\begin{equation}\label{d1c}
\norm{\mu_J}:=\sum_{t_{j_1},\dots,t_{j_\ell}\in\ooi}\norm{\mu_{J;t_{j_1},\dots,t_{j_\ell}}}
<\infty.
\end{equation}

Furthermore,
\begin{equation}\label{d1d}
2^{-m}  \norm{\chi} \le \sum_{J\subseteq\mm}\norm{\mu_J} \le 3^m\norm\chi.
\end{equation}

Conversely, every such family of signed Borel measures 
$\mu_{J;t_{j_1},\dots,t_{j_\ell}}$ satisfying \eqref{d1c} defines a
continuous linear functional on $\doim$ by \eqref{d1a}--\eqref{d1b}.
\end{theorem}

Note that the sum in \eqref{d1b} formally is uncountable (when $\ell>0$),
but \eqref{d1c} implies that $\mujt$ is non-zero only for a countable set of 
$(t_{j_1},\dots,t_{j_\ell})$, so all sums are really countable.

Note also that for $J=\emptyset$, with $\ell=0$, \eqref{d1b} reduces to
\begin{equation}\label{d1b0}
  \chi_\emptyset(f)=
\int_{I^{m}}  f(t_1,\dots,t_m)  \dd\mu_{\emptyset}(t_{1},\dots,t_{m}),
\end{equation}
so this term in \eqref{d1a} is simply $\int f\dd\mu_\emptyset$.
For the other extreme, $\ell=m$, 
$\mu_{\mm;t_1,\dots,t_m}$ is a signed measure on the one-point space $I^0$, \ie, a
real number, and \eqref{d1b} is interpreted as
\begin{equation}\label{d1bm}
  \chi_{\mm}(f)=\sum_{t_{1},\dots,t_{m}\in\ooi}
\mu_{\mm;t_{1},\dots,t_{m}}
\gD_{\mm} f(t_1,\dots,t_m) ,
\end{equation}
where 
$\norm{\mu_{\mm}}:=\sum_{t_1,\dots,t_m}|\mu_{\mm;t_{1},\dots,t_{m}}|<\infty$
and, again, the sums really are countable.
(In other words, $\mu_{\mm}=(\mu_{\mm;t_{1},\dots,t_{m}})$ is an element of
$\ell^1(\oi^m)$.)

\begin{proof}
  Since $\doim\cong C(\hI^m)$ (isometrically), we can use the Riesz
  representation theorem \refP{Priesz} and represent $\chi$ by a signed
Baire measure $\mu$ on $\hI^m$. We use the projections in \refS{Sproj} and 
expand $\mu$ as
\begin{equation}\label{orange}
  \mu=(\Phi_1+\Psi_1)\dotsm(\Phi_m+\Psi_m)\mu
=
\sum_{J\subseteq\mm} \tmu_J,
\end{equation}
where
\begin{equation}\label{muJ}
  \tmu_J:=\Bigpar{\prod_{i\notin J} \Phi_i\prod_{j\in J} \Psi_j }(\mu).
\end{equation}
We define  $\chi_J(f):=\int f\dd\tmu_J$; then \eqref{d1a} holds by
\eqref{orange}, and we proceed to show the representation \eqref{d1b}.
If $j\in J$, then $\Psi_j(\tmu_J)=\tmu_J$, and thus by \refL{LPsi},
$\tmu_J$ is supported on the set $\pi_j\qw(\hA_j)$ for some countable sets
$A_j\subset\ooi$.
 
Suppose for notational convenience that $J=\set{1,\dots,\ell}$.
Then, $\tmu_J$ is thus supported on 
\begin{equation}\label{white}
\bigcap_{j=1}^\ell\pi_j\qw(\hA_j)=\hA_1\times\dots\times\hA_\ell\times\hI^{m-\ell}.
\end{equation}

For $\xt_1,\dots,\xt_\ell\in\hI$,
let
$F_{\xt_1,\dots,\xt_\ell}:=
\set{\xt_1}\times\dots\times\set{\xt_\ell}\times\hI^{m-\ell}$ and
let $\tmu_{J;\xt_1,\dots,\xt_\ell}$ be the restriction of $\tmu_J$ to
$F_{\xt_1,\dots,\xt_\ell}$.

Fix some $(t_1,\dots,t_\ell)\in \ooi^\ell$,
and let
$\trho:F_{t_1,\dots,t_\ell}\to I^{m-\ell}$ be the map
$\trho(t_1,\dots,t_\ell,\xt_{\ell+1},\dots,\xt_m)
:=(\rho(\xt_{\ell+1}),\dots,\rho(\xt_m))$.
Let 
\begin{equation}\label{mauve}
\mu_{J;t_{j_1},\dots,t_{j_\ell}}:=
\trhox(\tmu_{J;t_{j_1},\dots,t_{j_\ell}})  
\end{equation}
be the signed measure on $I^{m-\ell}$ induced by this map,
noting that $\mujt=0$ unless
$(t_1,\dots,t_\ell)\in A_1\times\dots\times A_\ell$.

If $i\in\set{\ell+1,\dots,m}$, then  $\Phi_i(\tmu_J)=\tmu_J$ by \eqref{muJ}, 
and thus
\eqref{b4} implies that for any bounded Baire measurable function $f$ on
$F_{t_1,\dots,t_\ell}$, recalling \eqref{phii}, 
that
$\phi(\xt)=\iota(\rho(\xt))=\rho(\xt)$ for all $\xt\in\hI$ by \eqref{phi},
and \eqref{mauve},
\begin{align}\label{green}
\int_{F_{t_1,\dots,t_\ell}}&f \dd\tmu_J 
=\int_{\hI^m}\bigpar{f\etta_{F_{t_1,\dots,t_\ell}}}
\circ\phi_{\ell+1}\circ\dotsm\circ\phi_m \dd\tmu_J
\notag\\&
=\int_{F_{t_1,\dots,t_\ell}}
  f\bigpar{t_1,\dots,t_\ell,\rho(\xt_{\ell+1}),\dots,\rho(\xt_m)}
\dd\tmu_J
\notag\\&
=\int_{I^{m-\ell}}f(t_1,\dots,t_\ell,t_{\ell+1},\dots,t_m)
\dd\mujt(t_{\ell+1},\dots,t_m).  
\end{align}
Consider now any sequence $(\xt_i)_1^\ell$ with $\xt_i\in\set{t_i,t_i-}$,
$i=1,\dots,\ell$,  and let $q:=|\set{i:\xt_i=t_i-}|$.
For $i\le\ell$, \eqref{muJ} implies $\Psi_i(\tmu_J)=\tmu_J$,
and thus by \eqref{Psi+-} and \eqref{green}, for any 
bounded Baire measurable function $f$ on $F_{\xt_1,\dots,\xt_\ell}$,
writing $\vt'=(t_{\ell+1},\dots,t_m)$,
\begin{equation}\label{black}
  \begin{split}
\int_{F_{\xt_1,\dots,\xt_\ell}} f\dd\tmu_J
= (-1)^q \int_{I^{m-\ell}} f(\xt_1,\dots,\xt_\ell,\vt')    
\dd\mujt(\vt')
  \end{split}
.\end{equation}

Let 
$E_{t_1,\dots,t_\ell}:=\wset{t_1}\times\dots\times\wset{t_\ell}\times\hI^{m-\ell}$.
Summing \eqref{black} for the $2^\ell$ choices of $(\xt_1,\dots,\xt_\ell)$,
we obtain, recalling \eqref{gDJ},
\begin{equation}\label{blue}
  \begin{split}
\int_{E_{t_1,\dots,t_\ell}} f\dd\tmu_J
= \int_{I^{m-\ell}}\gD_J f(t_1,\dots,t_\ell,\vt') \dd\mujt(\vt').
  \end{split}
\end{equation}
Furthermore, recalling that $\tmu_J$ is supported on \eqref{white},
\begin{equation}
\chi_J(f)=
  \int_{\hI^m} f\dd\tmu_J 
= \sum_{(t_1,\dots,t_\ell)\in A_1\times\dots\times A_\ell} 
\int_{E_{t_1,\dots,t_\ell}} f\dd\tmu_J,
\end{equation}
and \eqref{d1b} follows by
summing \eqref{blue} 
over all $(t_1,\dots,t_\ell)\in A_1\times\dots\times A_\ell$.

Next,
the first inequality in \eqref{d1d} follows from \eqref{d1a}--\eqref{d1c},
noting $\norm{\gD_J f}\le 2^{|J|}\norm{f}\le 2^m\norm{f}$.
Furthermore,
$\norm{\mu_J}\le\norm{\tmu_J}$ by \eqref{d1c} and \eqref{mauve}, and
$\norm{\tmu_J}\le 2^{|J|}\norm{\mu}$ by \eqref{muJ} and
\eqref{normPhi}--\eqref{normPsi}.
Hence, the second inequality in \eqref{d1d} follows, noting that
$\sum_J2^{|J|}=3^m$ by the binomial theorem.

The converse, that every family satisfying \eqref{d1c} defines a continuous
linear functional on $\doim$ by \eqref{d1a}--\eqref{d1b} is obvious.
Furthermore, it is easily seen that if $\chi$ is defined in this way, then 
with $\tmu_J$ defined by \eqref{muJ}, $\int f\dd\mu_J$ equals the summand
$\chi_J(f)$ given by \eqref{d1b}, since the contribution from each
$\chi_{J'}$ with $J'\neq J$ vanishes by cancellations, 
and thus the construction above recovers
the measures $\mujt$ used to define $\chi$. In other words, the measures
$\mujt$ are uniquely determined by $\chi$.
\end{proof}

\section{Measurability and random variables in $\doim$}
\label{Smeas}

We equip $\doim$ with the \gsf{} $\cD=\cD_m$
generated by all point evaluations
$f\mapsto f(t)$, $t\in\oi^m$.
We sometimes mention this \gsf{}  explicitly for emphasis,
but even when no \gsf{} is mentioned, $\cD$ is implicitly assumed.

A $\doim$-valued \rv, or equivalently
a random element of $\doim$, is thus a function $X:\gO\to\doim$, defined on
some probability space $(\gO,\cF,P)$ such that for each fixed $t\in\oi^m$,
$X(t)$ is measurable (i.e., a random variable). 

Note that the norm $f\mapsto\norm{f}$ is a $\cD$-measurable function
$\doim\to\bbR$, since it suffices to take the supremum in \eqref{norm} over
rational $t$. Hence, if $X$ is a  $\doim$-valued \rv,
then $\norm X$ is measurable, \ie, a random variable.

\begin{remark}\label{RcD}
  $\cD$ is \emph{not} equal to the Borel \gsf{} defined by the norm
  topology on $\doim$, see \eg{} \cite[Example 2.5]{SJ271}.
The same example shows also that $\cD$ is strictly weaker than 
Borel \gsf{} defined by the weak topology. (We omit the details.)

However, in the positive direction, 
\refC{CM2} below shows that $\cD$ coincides with the \gsf{} generated by the
continuous linear functionals.

As a consequence (or directly), if we identify $\doim$ and $C(\hI^m)$ as
usual, then $\cD$ is also generated by all point evaluations
$f\mapsto f(\xt)$, $\xt\in\hI^m$.

Moreover, $\cD$ coincides also with the Borel \gsf{} 
defined by the Skorohod topology on $\doim$,
see \cite{Neuhaus}.
\end{remark}

\begin{theorem}
  \label{TM}
Let $m\ge1$ and $\ell\ge1$.
Every bounded multilinear form 
$\gY:\bigpar{\doim}^\ell\to\bbR$ is measurable.
\end{theorem}

For $m=1$, this is \cite[Theorem 9.19]{SJ271}.
Instead of trying to generalize the proof in \cite{SJ271}, we proceed
through a different route, using the known case $m=1$
and \refL{LM1} below (proved using several preliminary  lemmas).

We first record the important special case $\ell=1$; for $m=1$ this was
proved by \citet{Pestman}.

\begin{corollary}\label{CM1}
  Every continuous linear functional on $\doim$ is measurable.
\nopf
\end{corollary}

\begin{corollary}\label{CM2}
  The \gsf{} $\cD$ on $\doim$ coincides with the \gsf{} $\cB_w$ generated
  by the continuous linear functionals.
\end{corollary}
\begin{proof}
  \refC{CM1} implies that $\cB_w\subseteq \cD$. The converse follows because
  every point evaluation is a continuous linear functional.
\end{proof}

 We also rephrase this in terms of $\doim$-valued random
variables.
A function $X$ from a measure space into a Banach space $B$ is
\emph{weakly measurable} if $\innprod{\chi,X}$ is
measurable for every $\chi\in B^*$.

\begin{corollary}\label{CM3}
   If $X:\gO\to\doim$ is a function defined on some probability space
$(\gO,\cF,P)$, then $X$ is $\cD$-measurable 
(i.e., a random variable in $\doim$) if
and only if it is weakly measurable.
\nopf
\end{corollary}

We begin the proof of \refT{TM} by a simple observation. 
(See \cite[Lemma 9.12]{SJ271} for the case $m=1$.)
\begin{lemma}
  \label{LM0}
The evaluation map $(f,t_1,\dots,t_m)\mapsto f(t_1,\dots,t_m)$ is measurable
$\doim\times\oi^m\to\bbR$. 
\end{lemma}
\begin{proof}
By right-continuity,
  \begin{equation}
    f(t_1,\dots,t_m)
=
\lim_{\ntoo}
f\biggpar{\frac{\ceil{nt_1}}n,\dots,\frac{\ceil{nt_m}}n},
  \end{equation}
where the function on the \rhs{} 
is measurable for each fixed $n$.
\end{proof}

In the next lemmas we fix $f\in\doim$ and consider differences along
one coordinate only; for notational convenience we consider the first
coordinate and write $t=(t_1,t')$ with $t_1\in\oi$ and $t'\in\oi^{m-1}$.
Furthermore, to avoid some trivial modifications at the endpoints 0 and 1, we
extend $f$ by defining $f(t_1,t'):=f(0,t')$ for $t_1<0$
and 
$f(t_1,t'):=f(1,t')$ for $t_1>1$.

We define, recalling \eqref{gDi}.
\begin{equation}\label{gdxf}
  \gdxf(t_1):=\sup_{t'\in\oi^{m-1}}\bigabs{\gD_1f(t_1,t')},
\qquad t_1\in\oi,
\end{equation}
and, for an interval $J$,
\begin{equation}\label{glf}
  \glf(J):=\sup\bigset{|f(t_1,t')-f(u_1,t')|:t_1,u_1\in J,\,t'\in\oimi}.
\end{equation}

\begin{lemma}
\label{LM2}  
For every $t_1\in\oi$,
$f(s_1,t')\to f(t_1,t')$ as $s_1\downto t_1$ and $f(s_1,t')\to f(t_1-,t')$
as $s_1\upto t_1$, 
uniformly for all $t'\in\oi^{m-1}$.

In other words, $\glf([t_1,t_1+\gd])\to0$ and $\glf([t_1-\gd,t_1))\to0$ as
$\gd\to0$. 
\end{lemma}

\begin{proof}
A standard compactness argument.
  Let $\eps>0$.
  By the definition of \doim, for every $t=(t_1,t')\in\oi^m$, there
  exists an open 
  ball $B_t=B(t,\gd_t)$ centred at $t$ such that if $s\in B_t$, then $f(s)$
  differs by 
  at most $\eps/4$ from the limit as $s\to t$ in the corresponding octant.
It follows that if $(s_1,t')\in B_t$, then 
$|f(s_1,t')-f(t_1-,t')|\le\eps/2$ if $s_1<t_1$
and $|f(s_1,t')-f(t_1,t')|\le\eps/2$ if $s_1\ge t_1$.

Fix $t_1$.
By compactness, there exists a finite set 
$\set{\ttx 1,\dots\ttx N}$ 
such that the
corresponding balls $B_{(t_1,\ttx j)}$ cover $\set{t_1}\times\oimi$, and
furthermore, 
there exists 
$\gd>0$ such that for every $t'\in\oimi$, the ball $B((t_1,t'),\gd)$ is
contained in some $B_{(t_1,\ttx j)}$.
It follows that for any $t'\in\oimi$, if $s\in(t-\gd,t)$, then
$|f(s_1,t')-f(t_1-,t')|\le\eps/2$,
and if $s_1\in[t_1,t_1+\gd)$, then 
$|f(s_1,t')-f(t_1,t')|\le\eps/2$.
\end{proof}

Fix $\eps>0$ and let
\begin{equation}\label{xife}
  \xife:=\set{t\in\oi:\gdxf(t)\ge\eps}.
\end{equation}

\begin{lemma}\label{LM3}
The set $\xife$ is finite for every $f\in\doim$ and $\eps>0$.
\end{lemma}

This lemma is essentially the same as \cite[Lemma 1.3]{Neuhaus}.

\begin{proof}
Let the balls $B_t$ be as in the proof of \refL{LM2}.
It follows that if $s\in B_t$ and $s_1\neq t_1$, then $|\gD_1 f(s)|\le\eps/2$.

By compactness, there exists a finite set $\set{t^1,\dots,t^N}$ such that the
corresponding balls $B_{t^j}$ cover $\oi^m$. It follows that
$\xife$ is a subset of the finite set $\set{t^j_1:j\le N}$.
\end{proof}

Say that an interval $J\subseteq\oi$ is \emph{fat} if $\glf(J)\ge\eps$ and
\emph{bad} if $J$ is fat and furthermore $J\cap\xife=\emptyset$.

\begin{lemma}
 \label{LM4}
For every $f\in\doim$ and $\eps>0$, there exists $\eta>0$ such that
if $J\subset\oi$ is an interval of length $|J|<\eta$, then $J$ is not bad.
\end{lemma}

\begin{proof}
We claim that for every $t\in\oi$, there exists an open interval $O_t\ni t$
such that no  interval $J\subseteq O_t$ is bad.

In order to show this claim, suppose first that $t\notin\xife$. Then
$\gdxf(t)<\eps$, and it follows by \refL{LM2} that we can choose $\gd>0$
such that $\glf\bigpar{(t-\gd,t+\gd)}<\eps$. Hence, $O_t:=(t-\gd,t+\gd)$
contains no fat interval, and thus no bad interval.

On the other hand, if $t\in\xife$, we similarly see by \refL{LM2} that
we can choose $\gd>0$ such that $\glf\bigpar{(t-\gd,t)}<\eps$
and $\glf\bigpar{(t,t+\gd)}<\eps$.
Let $O_t:=(t-\gd,t+\gd)$.
Any interval $J\subseteq O_t$ either contains $t\in\xife$,
or it is a subset of $(t-\gd,t)$ or $(t,t+\gd)$ and then $J$ is not fat; in
both cases $J$ is not bad.

This proves the claim.
By a standard compactness argument (Lebesgue's covering lemma),
there exists $\eta>0$ such that every interval $J\subset\oi$
of length $|J|<\eta$ is contained  in some $O_t$, and thus not bad.
\end{proof}

\begin{lemma}
  \label{LM5}
Fix $\eps>0$, let $M:=|\xife|\ge0$ and write $\xife=\set{\xi_1,\dots,\xi_M}$
with $0<\xi_1<\dots<\xi_M\le1$. Let further $\xi_i:=0$ for $i>M$.
Then $M$ and all $\xi_i$, $i\ge1$, are measurable functionals of $f\in\doim$.
\end{lemma}

\begin{proof}
For $n\ge0$ and $1\le j\le 2^n$, let $J_{n,j}$ be the dyadic interval
$((j-1)2^{-n},j2^{-n}]$. 

  By \refL{LM4}, if $n$ is large enough (depending on $f$), then
$J_{n,j}$ is not bad; hence, if $J_{n,j}$ is fat, then
$J_{n,j}\cap\xife\neq\emptyset$. 

Conversely, if $J_{n,j}\cap\xife\neq\emptyset$, let $t\in J_{n,j}\cap\xife$;
then
\begin{equation}
  \glf(J_{n,j})\ge\gdxf(t)\ge\eps.
\end{equation}
Hence, for large $n$, $J_{n,j}$ contains some $\xi_i\in\xife$ if and only if
$J_{n,j}$ is fat.

Moreover, since $\xife$ is finite by \refL{LM3},
if $n$ is large enough, then each $J_{n,j}$ contains at most one point
$\xi_i$.

For $n\ge0$, suppose that $q_n$ of the intervals $J_{n,j}$, $1\le j\le
2^n$, are fat, and let these be $J_{n,j_i}$, $i=1,\dots,q_n$, with
$j_1<\dots<j_{q_n}$. 
Let further 
\begin{equation}
  \xi_{ni}:=
  \begin{cases}
j_{ni}/2^n, & i\le q_n,    
\\
0, & i>q_n.
  \end{cases}
\end{equation}
We have shown above that for large $n$, $q_n=M$. Hence, $M=\lim_{\ntoo}
q_n$.
Moreover, it follows from the argument above that for each fixed $i\ge1$,
$\xi_{in}\to\xi_i$ as \ntoo.

Since each $\glf(J_{n,j})$ is a measurable functional of $f$ 
(because it suffices to
take the supremum in \eqref{glf} over rational $t,u,t'$), it follows that
each $q_n$ and $\xi_{ni}$ is measurable, and thus so are their limits $M$
and $\xi_i$.
\end{proof}

If $F$ is a finite subset of $\oi$, arrange the elements of $F\cup\setoi$ as
$0=x_0<x_1<\dots<x_N=1$ and define
\begin{equation}  \label{glxf}
\glxf(F):=\max_{1\le i\le N}\glf\bigpar{[x_{i-1},x_i)}
=\max_{1\le i\le N}\glf\bigpar{(x_{i-1},x_i)},
\end{equation}
where the last equality holds by the right-continuity of $f$.

\begin{lemma}\label{LM6}
For every  $f\in\doim$, there exists a sequence $(\xxi_j)_j^\infty$ in $\oi$
such that
\begin{equation}\label{lm6}
  \glxf\bigpar{\set{\xxi_j}_{j=1}^n} \to 0 
\qquad\text{as \ntoo}.
\end{equation}
Moreover, these points can be chosen such that each $\xxi_j$, $j\ge1$, is a
measurable functional of $f\in\doim$.
\end{lemma}

\begin{proof}
For each $k\ge1$, let $\xxxi_{ki}$, $i\ge1$, be the numbers defined in
\refL{LM5} for $\eps=1/k$. Then each $\xxxi_{ki}$ is a measurable functional
of $f$.
Consider all these functionals for $k\ge1$ and $i\ge1$, together with the
constant functionals $r$ for every rational $r\in\oi$, and arrange this
countable collection of functionals in a sequence $\xxi_j$, $j\ge1$
(in an arbitrary but fixed way, not depending on $f$). 

Now suppose that $f\in\doim$, and 
let $F_n:=\set{\xxi_j}_1^n$.
Let $k\ge1$, and let $\eps=1/k$. 
Then $M:=|\xife|<\infty$ by \refL{LM3}, 
and thus
there exists $n_1$ such that if $i\le M$, then $\xxxi_{ki}=\xxi_j$ for some
$j\le n_1$. Hence, if $n\ge n_1$, then 
\begin{equation}\label{n1}
  F_n\supseteq\xife.
\end{equation}

Furthermore, let $\eta$ be as in \refL{LM4}, and let $L:=\floor{1/\eta}+1$.
Since the rational numbers $p/L$, $0\le p\le L$ all appear as some
$\xxi_i$, it follows that there exists $n_2$ such that if $n\ge n_2$ then
$F_n\supseteq\set{p/L}_{p=0}^L$. Hence, if $F_n\cup\setoi=\set{x_i}_1^N$ as in
\eqref{glxf}, then each interval $(x_{i-1},x_i)$ has length $x_i-x_{i-1}\le
1/L<\eta$. Consequently, by \refL{LM4}, the interval is not bad.
Moreover, if  also $n\ge n_1$, then \eqref{n1} holds and thus
$(x_{i-1},x_i)\cap\xife=\emptyset$. Consequently, for $n\ge \max(n_1,n_2)$,
no interval $(x_{i-1},x_i)$ is fat, and thus \eqref{glxf} yields
$\glxf(F_n)<\eps=1/k$. 

Since $k$ is arbitrary, this shows \eqref{lm6}.
\end{proof}

As noted in \refS{SStensor}, 
linear combinations of functions of the form
$\bigotimes_1^m f_i=\prod_{i=1}^m f_i(x_i)$ with $f_i\in\doi$
are dense in $\doim$.
The next lemma shows that $f\in\doim$ can be approximated by such linear
combinations in a measurable way.

\begin{lemma}
\label{LM1}
For every $f\in\doim$, there exist functions $f_{N,k,i}\in\doi$
for $N\ge1$, $1\le k\le N$ and $1\le i\le m$ such that
\begin{equation}\label{lm1}
f_N:= \sum_{k=1}^N 
\bigotimes_{i=1}^m f_{N,k,i}
\to f
\qquad
\text{in }\doim
\end{equation}
(i.e., uniformly), as \Ntoo.
Furthermore, the functions 
$f_{N,k,i}$ can be chosen 
such that the mappings
$f\mapsto f_{N,k,i}$ are measurable $\doim\to\doi$.
\end{lemma}

\begin{proof}
We have so far considered the first coordinate. Of course, the results above
hold for any coordinate. 
We let $\glif(J)$ and $\glixf(F)$ be defined as in 
\eqref{glf} and \eqref{glxf}, but using the $i$-th coordinate instead of the
first. 
Thus \refL{LM6} shows that for every $i\le m$, there exists a sequence
of measurable functionals
$\xi^i_j$, $j\ge1$, such that
\begin{equation}\label{lm6i}
  \glixf\bigpar{\set{\xi^i_j}_{j=1}^n} \to 0 
\qquad\text{as \ntoo}.
\end{equation}
 
For $n\ge0$ and $1\le i\le m$, 
arrange $\set{\xi^i_j}_{j=1}^n\cup\setoi$ in increasing order
as $0=x^i_0<\dots<x^i_{n_i}=1$, where $n_i\le n+1$. 
(Strict inequality is
possible because there may be repetitions in $\set{\xi^i_j}_{j=1}^n\cup\setoi$.)
Let $J^i_j:=[x^i_{j},x^i_{j+1})$ for $j<n_i$ and $J^i_{n_i}:=\set{1}$. Thus
$\set{J^i_j}_{j=0}^{n_i}$ is a partition of $\oi$.
Let $h^i_j:=\etta_{J^i_j}$, the indicator function of $J^i_j$.
(Note that $J^i_j$ and $h^i_j$ depend on $n$.)

Now define the step function $g_n$ on $\oi^m$ by
\begin{equation}\label{gn1}
  g_n:= \sum_{j_1,\dots,j_m}f(x^1_{j_1},\dots,x^m_{j_m}) 
\tensorim h^i_{j_i},
\end{equation}
i.e.,
\begin{equation}\label{gn2}
  g_n(t_1,\dots,t_m):= f(x^1_{j_1},\dots,x^m_{j_m}) 
\qquad\text{when $t_i\in J^i_{j_i}$ ($1\le i\le m$)}.
\end{equation}
It follows from the definitions \eqref{glf} and \eqref{glxf} that if 
$t_i\in J^i_{j_i}$ for every $i$, then
\begin{equation}
  |g_n(t_1,\dots,t_m)-f(t_1,\dots,t_m)|
\le\sumim \glif(J_{j_i}^i)
\le\sumim\glixf(\set{\xi^i_j}_{j=1}^n).
\end{equation}
Hence, \eqref{lm6i} implies that
\begin{equation}
  \norm{g_n-f}=\sup_{t\in\oi^m}|g_n(t)-f(t)|\to0
\qquad
\text{as \ntoo},
\end{equation}
i.e., $g_n\to f$ in $\doim$.

The rest is easy. We can write \eqref{gn1} as
\begin{equation}\label{gn3}
  g_n=\sum_{j_1,\dots,j_m}\tensorim g_{n,j_1,\dots,j_m,i}
\end{equation}
with
\begin{equation}\label{gnjji}
  g_{n,j_1,\dots,j_m,i}:=
\begin{cases}
f(x^1_{j_1},\dots,x^m_{j_m})h^1_{j_1} ,& i=1, 
\\
h^i_{j_i}  , & i>1.
\end{cases}
\end{equation}
The sum in \eqref{gn3} has $\prod_i (n_i+1)\le (n+2)^m$ terms; by rearranging
the terms in lexicographic order of $(j_1,\dots,j_m)$, 
we may write it as
\begin{equation}\label{gn4}
  g_n=\sum_{k=1}^{(n+2)^m}\tensorim g_{n,k,i},
\end{equation}
where we, if necessary, have added terms that are 0 (with all $g_{n,k,i}=0$).
Finally, we relabel again, defining for $(n+2)^m\le N<(n+3)^m$
\begin{equation}
f_{N,k,i}:=
  \begin{cases}
    g_{n,k,i}, & k\le(n+2)^m,
\\
0, & k>(n+2)^m.
  \end{cases}
\end{equation}
Then $f_N$ defined by \eqref{lm1} satisfies $f_N=g_n$ for 
$(n+2)^m\le N<(n+3)^m$, and thus $f_N\to f$ in $\doim$ as \Ntoo.

It is clear from the construction above that every $n_i$ and $x^i_j$ is a
measurable functional of $f$; using \refL{LM0} it follows that
every $g_{n,j_1,\dots,j_m,i}$ defined by
\eqref{gnjji} depends measurably on $f$, and thus so does every $g_{n,k,i}$
and every $f_{N,k,i}$.
\end{proof}

\begin{remark}
  The proof above yields functions $f_{N,k,i}$ of the special
form $a\etta_{[b,c)}$
or $a\etta_{\set1}$, where $a,b,c$ are measurable functionals of $f$.
Cf.\ \cite{Wichura-PhD,Wichura1969}. 
\end{remark}

\begin{proof}[Proof of \refT{TM}]
We use \refL{LM1}, with some fixed measurable choice of $f_{N,k,i}$.
For every $\ell$-tuple $(f^1,\dots,f^\ell)$, we apply \refL{LM1} to each
$f^j$
and obtain, by continuity and multilinearity of $\gY$,
\begin{equation}
  \label{ups}
  \begin{split}
\gY\bigpar{f^1,\dots,f^\ell}
&= \lim_{\Ntoo} \gY\bigpar{f_N^1,\dots,f^\ell_N}
\\&
= \lim_{\Ntoo}\sum_{k_1,\dots,k_\ell=1}^N \gY
 \Bigpar{\bigotimes_{i=1}^m f^1_{N,k_1,i},\dots,\bigotimes_{i=1}^m f^\ell_{N,k_\ell,i}}
.  
\end{split}
\end{equation}

Define a bounded $\ell m$-linear form $\tgY$ on $\doi$ by
\begin{equation}
  \tgY\bigpar{g_{11},\dots,g_{1m},\dots,g_{\ell1},\dots,g_{\ell m}}
:=
\gY\Bigpar{\bigotimes_{i=1}^m g_{1,i},\dots,\bigotimes_{i=1}^m g_{\ell,i}}.
\end{equation}
Then the summand in \eqref{ups} is
$\tgY\bigpar{(f^j_{N,k_j,i})_{1\le j\le \ell,\,1\le i\le m}}$.
We apply the case $m=1$ of the theorem,
which as said above is \cite[Theorem 9.19]{SJ271}, 
to $\tgY$ (with $\ell$ replaced by $\ell m$);
since each $f^j\mapsto f^j_{N,k_j,i}$ is measurable, 
this
shows that
each summand is a measurable function of 
$(f^1,\dots,f^\ell)\in\xpar{\doim}^\ell$.
Hence, so is their sum in \eqref{ups}, 
and thus by \eqref{ups}, also $\gY(f^1,\dots,f^\ell)$.
\end{proof}

\section{A Fubini theorem}\label{SFubini}

Recall that a $\doim$-valued random variable $X$ is a measurable function
$X:\ofp\to\doim$ for some (usually unspecified) probability space $\ofp$;
hence $X$ can be regarded as a function 
$X(\go,t):\gO\times\oi^m\to\bbR$,
and the measurability condition means that $X(\cdot,t)$ is measurable for
each fixed $t\in\oi^m$.
In fact, $X(\go,t)$ is jointly measurable on $\gO\times\oi^m$ as a
consequence of \refL{LM0}.

Since functions in $\doim$ extend uniquely to $\hI^m$, 
yielding an natural identification 
$\doim\cong C(\hI^m)$, we can also regard the random variable $X$ as a
function $X:\ofp\to\chIm$ and thus also as a function
$X(\go,\xt):\gO\times\him\to\bbR$. 
This function is measurable in $\go$ for every fixed $\xt\in\him$ by \refC{CM1}
(or more simply by considering a sequence $t_n\in\oi^m$ that converges to
$\xt$ in $\him$); it is also a continuous function of $\xt$ and thus Baire
measurable for every fixed $\go$. In other words, the function $X(\go,\xt)$
 is separately measurable. 
However, $X(\go,\xt)$ is in general \emph{not} jointly mesurable on
$\gO\times\him$.
In fact, the example in \cite[Remark 9.18]{SJ271} 
shows that \refL{LM0} does not extend to the evaluation map 
$\chIm\times\him$, and we may then choose $\gO=\chIm$ with $X$ the identity.
(Here it does not matter whether we consider Baire or Borel measurability on
$\him$.) 

This lack of joint measurability is a serious technical problem.
A continuous linear functional on $\chIm$ is given by 
integration with respect to 
a Baire measure $\mu$
on $\him$, see \refP{Priesz}, and we would like to be able
to use Fubini's theorem and interchange to order of integrations with respect
to $\mu$ and the probability measure $\P$ on $\gO$, but
the lack of joint measurability means that a straight-forward application of
Fubini's theorem is not possible. However, the following theorem  shows that
the desired result nevertheless holds.

We say that a function $f$ on a measure space $(S,\cS,\mu)$ is
\emph{$\mu$-measurable} if it is defined $\mu$-\aex{} and is
$\mu$-\aex{} equal to an $\cS$-measurable
function (this is equivalent to $f$ being measurable with respect to the
$\mu$-completion of $\cS$). 
Furthermore, $f$ is \emph{$\mu$-integrable} if it is
$\mu$-measurable and $\int |f|\dd|\mu|<\infty$.
(Recall that $|\mu|$ denotes the  variation measure of $\mu$.)

\begin{theorem}
  \label{Tfubini}
Suppose that  $X$ is a random variable in $\doim=C(\hI^m)$, defined on some
probability space $\ofp$,
and that
$\mu\in\MBa(\hI^m)$ is a signed Baire measure on $\hI^m$.
\begin{romenumerate}
\item \label{TfubiniC}
If\/
$\norm{X}\le C$ for some constant $C<\infty$, \ie, $|X(\go,\xt)|\le C$
for all 
$\go\in\gO$ and $t\in\oi^m$,
then
$\go\mapsto \inthm X(\xt)\dd\mu(\xt)$ 
is a (bounded) measurable function on $\gO$, 
\ie, a random variable, 
and\/ $\xt\mapsto\E[X(\xt)]$ is an element of $\chIm$, and thus Baire
measurable on $\him$.
\item \label{Tfubini+}
If\/ $X\ge 0$ and $\mu$ is a positive Baire measure,
then
$\go\mapsto\inthm X(\xt)\dd\mu(\xt)$ is a measurable function
$\gO\to[0,\infty]$,  
\ie, a (possibly infinite) random variable,
and\/ $\xt\mapsto\E[X(\xt)]$ is a Baire measurable function $\him\to[0,\infty]$.
\item \label{TfubiniB}
If\/
$\E|X(t)|\le C$ for some constant $C<\infty$
and all $t\in\oi^m$,
then
$\inthm X(\xt)\dd\mu(\xt)$ exists \as{} and defines an integrable function
on $\gO$,  
\ie, an integrable random variable; 
and\/ $\xt\mapsto\E[X(\xt)]$ is a bounded Baire
measurable function on $\him$.
\item \label{Tfubini+-}
If either
$\E \inthm |X(\xt)|\dd|\mu|(\xt)<\infty$
or
$ \inthm \E[|X(\xt)|]\dd|\mu|(\xt)<\infty$,
then 
$\inthm X(\xt)\dd\mu(\xt)$ exists \as{} and defines an integrable function
on $\gO$,  
\ie, an integrable random variable; similarly,
$\E[X(\xt)]$ exists for $\mu$-\aex{} $\xt$ and defines a $\mu$-integrable 
function on $\him$.
\end{romenumerate}

In all four cases,
\begin{equation}
  \label{fubini}
\E \inthm X(\xt)\dd\mu(\xt)
=
\inthm \E[X(\xt)]\dd\mu(\xt).
\end{equation}

\end{theorem}

We first prove a simple lemma, which is useful also in other situations.

\begin{lemma}
  \label{Lfatou}
Suppose that $X$ is a random variable in $\doim$.
Then, for every  $\xt\in\him$, 
\begin{equation}\label{fatou1}
\E|X(\xt)| \le
\sup_{t\in\oi^m}\E|X(t)|.
\end{equation}
Consequently,
\begin{equation}\label{fatou2}
  \sup_{\xt\in\him}\E|X(\xt)| =
\sup_{t\in\oi^m}\E|X(t)|.
\end{equation}
\end{lemma}

\begin{proof}
If $\xt\in\him$, then there exists a
sequence $t_n\in\oi^m$ such that $t_n\to\xt$, and thus $X(t_n)\to X(\xt)$.
Hence, Fatou's lemma implies
\begin{equation}\label{fatou}
  \E |X(\xt)|\le \liminf_\ntoo\E|X(t_n)|
\le \sup_{t\in\oi^m}\E|X(t)|.
\end{equation}
This shows \eqref{fatou1}, and \eqref{fatou2} is an immediate consequence.
\end{proof}

\begin{proof}[Proof of \refT{Tfubini}]
  
\pfitemref{TfubiniC}
First, $\inthm X(\xt)\dd\mu(\xt)$ is measurable by \refC{CM3}, since
$\chi:f\mapsto\inthm f\dd\mu$ is a continuous linear functional on $\doim$.

Secondly, if $\xt_n\to\xt\in\him$, then $X(\xt_n)\to X(\xt)$ by continuity,
and thus 
$\E [X(\xt_n)]\to \E [X(\xt)]$ by dominated convergence. 
This shows that $\xt\mapsto \E [X(\xt)]$ is sequentially continuous, which is
equivalent to continuity since $\hI$ is first countable 
(see \refSS{SSsplit}).
Alternatively, considering only $t\in\oi^m$, dominated convergence shows
that 
$t\mapsto \E [X(t)]$ is a function in $\doim$, and that its continuous
extension to $\chIm$ is given by $\E [X(\xt)]$.

Finally, to show \eqref{fubini} we consider again the continuous linear
functional
$\chi:f\mapsto\inthm f\dd\mu$ and use the decomposition in \refT{Tdual}.
Fix $J\subseteq[m]$ and suppose for notational convenience that
$J=\set{1,\dots,\ell}$ for some $\ell\in\set{0,\dots,m}$.
(The cases $\ell=0$ and $\ell=m$ are somewhat special; we leave the
 simplifications in these cases to the reader.)
Also fix $t_1,\dots,t_\ell\in\ooi$ and consider the corresponding term
in \eqref{d1b}. Then $\gD_J f(t_1,\dots,t_m)$ is a linear combination of the
$2^\ell$ terms $f(\xt_1,\dots\xt_\ell,t_{\ell+1},\dots,t_m)$ with $\xt_j\in
\set{t_j,{t_j-}}\subset\hI$ for $i=1,\dots,\ell$.

Fix one such choice of $\xt_1,\dots,\xt_\ell$, and 
define for $f\in\chIm$ the function $\gam(f)$ on $\hI^{m-\ell}$ by
$\gam(f)(\xu_{1},\dots,\xu_{m-\ell})
:=f(\xt_1,\dots,\xt_\ell,\xu_1,\dots,\xu_{m-\ell})$; 
in other words, $\gam(f)$ is the restriction of $f$ to the
$(m-\ell)$-dimensional slice with the coordinates in $J$ fixed to 
$\xt':=(\xt_1,\dots,\xt_\ell)$, regarded as a function on $\hI^{m-\ell}$.
Obviously, $\gam(f)\in \chIml$ for every $f\in\chIm$, so
$\gam:\doim=\chIm\to\chIml=\doiml$.
Furthermore, for any fixed $u:=(u_1,\dots,u_{m-\ell})\in\oi^{m-\ell}$,
the mapping $f\mapsto \gam(f)(u)=f(\xt',u)$ is $\cD_m$-measurable on $\doim$, 
see \refR{RcD}; hence, $\gam:\doim\to\doiml$ is measurable for $\cD_m$ and
$\cD_{m-\ell}$. Hence, $\gam(X)$ is a $\doiml$-valued random variable, and
applying \refL{LM0} to $\gam(X)$, we see that
$
  (\go,u)\mapsto \gam(X)(\go,u)
=X(\go,(\xt',u))
$ 
is jointly measurable on $\gO\times\oi^{m-\ell}$. Consequently, we can use
Fubini's theorem and conclude
\begin{equation}
  \E \int_{I^{m-\ell}}X(\xt',u)\dd\mu_{J;t'}(u)
=
   \int_{I^{m-\ell}}\E X(\xt',u)\dd\mu_{J;t'}(u).
\end{equation}
Summing over all $2^\ell$ choices of $\xt'$ for a fixed
$t'=(t_1,\dots,t_\ell)$, after multiplying with the correct sign, we obtain,
letting $\E X$ denote the function $t\mapsto\E[X(t)]$ in $\doim$,
\begin{equation}\label{kq}
  \E \int_{I^{m-\ell}}\gD_J X(t',u)\dd\mu_{J;t'}(u)
=
   \int_{I^{m-\ell}}\gD_J(\E X)(t',u)\dd\mu_{J;t'}(u).
\end{equation}
We now apply the decomposition \eqref{d1b}
 and sum \eqref{kq} over all $t'\in\ooi^\ell$; 
recall that the sum really is countable.
This yields,
using \eqref{d1b} twice and
justifying the interchange of order of summation and expectation in the second
equality below by dominated convergence, because
\begin{align}
\lrabs{\int_{I^{m-\ell}}\gD_J X(t',u)\dd\mu_{J;t'}(u)}
\le 2^J C\norm{\mu_{J;t'}}  
\end{align}
where the sum over $t'$ of the \rhs{s} is finite by \eqref{d1c},
\begin{equation}\label{kkk}
  \begin{split}
\E[\chi_J(X)]&
= \E \sum_{t'\in\ooi^\ell}\int_{I^{m-\ell}}\gD_J X(t',u)\dd\mu_{J;t'}(u)     
\\&
=  \sum_{t'\in\ooi^\ell}\E\int_{I^{m-\ell}}\gD_J X(t',u)\dd\mu_{J;t'}(u) 
\\&
=  \sum_{t'\in\ooi^\ell}\int_{I^{m-\ell}}\gD_J(\E X)(t',u)\dd\mu_{J;t'}(u) 
\\&
= \chi_J(\E X),
  \end{split}
\end{equation}

Finally, summing \eqref{kkk} over $J\subseteq[m]$ yields, by \eqref{d1a},
$\E\chi(X)=\chi(\E X)$, which is \eqref{fubini}.

\pfitemref{Tfubini+}
Let $X_n(\go,\xt):=\min\bigpar{X(\go,\xt),n}$, and note that $X_n$ is a
random variable in $\chIm$ for every $n\ge1$. The result follows by applying
\ref{TfubiniC} to each $X_n$ and letting \ntoo, using the monotone
convergence theorem repeatedly.

\pfitemx{\ref{TfubiniB}, \ref{Tfubini+-}}
Note first that the two alternative conditions in \ref{Tfubini+-}
are equivalent by
\eqref{fubini} applied to $|X|$ and $|\mu|$, which is valid by 
\ref{Tfubini+}.
Furthermore, 
by \refL{Lfatou}, the assumption in \ref{TfubiniB} implies 
\begin{equation}
  \label{lz}
\E|X(\xt)|\le C, 
\qquad \xt\in\him, 
\end{equation}
which in turn implies the assumption in
\ref{Tfubini+-}.

Decompose $X=X_+-X_-$, where $X_+(\go,\xt):=\max\bigpar{X(\go,\xt),0}$
and $X_-(\go,\xt):=\max\bigpar{-X(\go,\xt),0}$, and note that $X_+$ and
$X_-$ are random variables in $\chIm$.
Similarly, use the Hahn decomposition
\cite[Theorem 4.1.4]{Cohn}
$\mu=\mu_+-\mu_-$ as the difference of two positive Baire measures on
$\chIm$.
The result follows by applying \ref{Tfubini+} to all combinations of $X_\pm$
and $\mu_\pm$, noting that this yields
$\E\inthm X_\pm(\xt)\dd\mu_\pm(\xt)=\inthm\E [X_\pm(\xt)]\dd\mu_\pm(\xt)<\infty$,
using the assumptions.
For \ref{TfubiniB}, note also that \eqref{lz} shows that $\E[X_\pm(\xt)]$ is
a bounded (and thus finite) function on $\him$.
\end{proof}

\section{Separability}\label{Ssep}

The Banach space $\doim$ is non-separable, which is a serious complication
in various ways already for $m=1$, see \eg{} \cite{Billingsley} and
\cite{SJ271}. 

Let $X:\ofp\to B$ is a function defined on a \ps{} and taking values in a
Banach space $B$. 
(In particular, $X$ may be a $B$-valued \rv, for a  given \gsf{} on $B$, but
here we do not assume any measurability.)
We then say, following \cite[Definition 2.1]{SJ271}, that
\begin{romenumerate}
\item \label{assep}
$X$ is \emph{\assep} if there exists a separable subspace $B_1\subseteq B$
  such that $X\in B_1$ a.s.
\item \label{wassep}
$X$ is \emph{\wassep} if there exists a separable subspace $B_1\subseteq B$
  such that if $\xx\in B\q$ and $\xx(B_1)=0$, then $\xx(X)=0$ a.s.
\end{romenumerate}

Note that in \ref{wassep}, the exceptional null set may depend on $\xx$.
(In fact, this is what makes the difference from \ref{assep}:
to assume $\xx(X)=0$ outside some fixed null set for all $\xx$ as in
\ref{wassep} is equivalent to \ref{assep}.)

\begin{remark}
  \label{Rassep}
A.s.\ separability is a powerful condition, which essentially reduces the
study of $X$ 
to the separable case.
Unfortunately, it is too strong for our purposes. 
In the case $m=1$, a random variable $X$ taking values in \doi{} is \assep{}
if and only if there exists a fixed countable set $N$ such that \as{} every
discontinuity point of $X$ belongs to $N$ \cite[Theorem 9.22]{SJ271}; we
extend this to \doim{} in \refT{Tassep} below.
Hence, 
in applications to random variables in $\doi$ or $\doim$,
this condition is useful only for variables that have a fixed set of
discontinuities, but  
not when there are discontinuities at random locations.
We therefore mainly use the weaker property '\wassep' defined in \ref{wassep}.
\end{remark}

\begin{example}
Let $U\sim \U(0,1)$ be a uniformly distributed random
  variable,
and let $X$ be the random element of $\doi$ given by $X=\etta_{[U,1]}$,
\ie{} $X(t)=\ett{U\le t}$. 
Then, see \cite[Example 2.5]{SJ271} for details, 
$X$ is not \assep, but
$X$ is \wassep. 
(We can take $B_1=\coi$ in the definition \ref{wassep} above.)
\end{example}

We note the following simple properties. 
\begin{lemma}
  \label{LAS}
Let $B$ be a Banach space, and assume that $X_1,X_2,\dots$ are \wassep{}
functions $\ofp\to B$ for some \ps{} $\ofp$.
  \begin{romenumerate}
  \item \label{LAS-sum}
Any finite linear combination $\sum_{i=1}^N a_iX_i$ is \wassep.
\item \label{LAS-lim}
If  $X_n\to X$ in $B$ a.s., then $X$ is \wassep.
\item \label{LAS-map}
If
$T:B\to \tilde B$ is a
bounded linear map into another Banach space,
then $T(X_1)$ is \wassep{} in $\tilde B$.
  \end{romenumerate}
The same properties hold for \assep{} functions too.
\end{lemma}

\begin{proof}
  Let $B_{1i}$ be a separable subspace of $B$ satisfying the property in the
  definition  for $X_i$, and let  $B_1$ be the closed
  subspace generated by $\bigcup_i B_{1i}$. Then $B_1$ is separable, and it
  is easily seen that this subspace verifies \ref{LAS-sum} and
  \ref{LAS-lim}.
For \ref{LAS-map} we similarly use $\tilde B_1:=T(B_1)$.
We omit the details.
\end{proof}

It was shown in \cite[Theorems 9.24 and 9.25]{SJ271} that random variables
in \doi{}   always are \wassep, 
and so are tensor powers of them in either the injective or projective
tensor power. We extend this to \doim.

\begin{theorem}
  \label{TDwassep}
  \begin{thmenumerate}
  \item \label{TDwassep1}
Let $X$ be a \dmeas{} 
$\doim$-valued \rv.
Then $X$ is \wassep.
\item \label{TDwassep2}
More generally, let 
 $X_1,\dots,X_\ell$ be \dmeas{} 
$\doim$-valued \rv{s}.
Then, 
$\bigtensor_{i=1}^\ell X_i$ 
is \wassep{} in the projective and injective tensor products 
$\doim\ptpx\ell$ and $\doim\itpx\ell$.
  \end{thmenumerate}
\end{theorem}

\begin{proof}
  Consider first $m=1$. Then, as said above, \ref{TDwassep1} is
  \cite[Theorem 9.24]{SJ271}, while
  \cite[Theorem 9.25]{SJ271} is the special case $X_1=\dots=X_\ell$ of
  \ref{TDwassep2}; moreover,  it is easily checked that the proof of
  \cite[Theorem 9.25]{SJ271} applies also to the case of general
  $X_1,\dots,X_\ell$.
(The main difference in the proof is that we fix a countable set $N$ 
and consider $X_1,\dots,X_\ell$ that \as{} are continuous at every fixed
$t\notin N$.)
Hence, the results hold for $m=1$.

In general, we apply \refL{LM1} to $X_1,\dots,X_\ell$ and conclude that
there are random variables 
$X^j_{N,k,i}$ in \doi{} such that,
 for every $j=1,\dots,\ell$,
\begin{equation}\label{lm1l}
X^j_N:= \sum_{k=1}^N 
\bigotimes_{i=1}^m X^j_{N,k,i}
\to X_j
\qquad
\text{in }\doim
\end{equation}
as \Ntoo.
Let
$  X_N:=X^1_N\tensor\dotsm\tensor X^\ell_N
\in \doim\ptpx\ell$.
Then $X_N\to X:=X^1\tensor\dotsm\tensor X^\ell$ 
in $\doim\ptpx\ell$
as \Ntoo.
Furthermore, by \eqref{lm1l},
\begin{equation}
  X_N
= \sum_{k_1,\dots,k_\ell\le N} 
\bigotimes_{j=1}^\ell\bigotimes_{i=1}^m X^j_{N,k_j,i}
.  
\end{equation}
By the case $m=1$ of the theorem (with $\ell$ replaced by $\ell m$),
each term in this sum is a \wassep{} random variable in 
$\doi\ptpx{\ell m}$.
Since the canonical inclusion $\doi\ptpx m\to\doi\itpx m=\doim$ is
continuous, and thus induces a continuous map
$\doi\ptpx{\ell m}=(\doi\ptpx m)\ptpx\ell\to\doim\ptpx\ell$,
it follows by \refL{LAS} that each $X_N$ is \wassep{} in 
$\doim\ptpx\ell$, and thus so is their limit $X$.

This proves the result for the projective tensor product.
For the injective \tp{} we use \refL{LAS}\ref{LAS-map} again,  with the
continuous inclusion
$\doim\ptpx \ell\to\doim\itpx \ell$.
\end{proof}

In contrast, and for completeness, we have the following characterization of
\assep{} random variables. (The case $m=1$ is \cite[Theeorem 9.22]{SJ271}.)
\begin{theorem}
  \label{Tassep}
Let $X$ be a \dmeas{} $\doim$-valued \rv.
Then
$X$ is \assep{} if and only if there exist (non-random) countable subsets
$\NA_1,\dots,\NA_m$ of\/ $\oi$ 
such that for every $i\le m$,
\as{} $\gD_i X(t)=0$ for all $t=(t_1,\dots,t_m)$
with $t_i\notin \NA_i$.
\end{theorem}

We consider first the deterministic case.
\begin{lemma}
  \label{Lassep}
Let $f\in\doim$.
Then
there exist countable subsets
$\NA_1,\dots,\allowbreak \NA_m$ of\/ $\oi$ 
such that for every $i\le m$,
$\gD_i f(t)=0$ for all $t=(t_1,\dots,t_m)$
with $t_i\notin \NA_i$.
\end{lemma}
\begin{proof}
  Consider the first coordinate and let,
recalling \eqref{gdxf} and \eqref{xife}, 
  \begin{equation}
\NA_1:=\set{t\in\oi:\gdxf(t)>0}=\bigcup_k\Xi_{f,1/k}.    
  \end{equation}
Then $\NA_1$ satisfied the claimed property by the definition \eqref{gdxf},
and $\NA_1$ is countable by \refL{LM3}. The same holds for $i>1$ by
relabelling the coordinates.
\end{proof}

\begin{proof}[Proof of \refT{Tassep}]
If $X$ is \assep, let $D_1$ be a separable subspace of $\doim$ such that
$X\in D_1$ a.s.
Let $\set{f_n}$ be a countable dense subset of $D_1$, and apply
\refL{Lassep} to $f_n$ for each $n$, yielding countable sets $\NA_{in}$.
Define $\NA_i:=\bigcup_n \NA_{in}$. 
Then, for every $i$,  
$\gD_i f(t)=0$ for all $t=(t_1,\dots,t_m)$ with $t_i\notin \NA_i$
and all $f\in D_1$; hence $\gD_i X(t)=0$ \as{} for all such $t$.

Conversely, suppose that such $\NA_1,\dots,\NA_m$ exist.
Then, using the notation in \refS{Smeas} (with $f=X$),
\as{} $\gdxX(t_1)=0$ for every $t\notin \NA_1$, and thus $\xiXe\subseteq \NA_1$.
Hence, the construction in the proof of \refL{LM6} yields $\xi_j\in
\tNA_1:=\NA_1\cup( \bbQ\cap\oi)$ \as{} for every $j$.
Consequently, in the proof of \refL{LM1}, 
\as{} every $\xi^i_j\in\tNA_i:=\NA_i\cup(\bbQ\cap\oi)$ and every $x^i_j\in \tNA_i$.
Let $Q_i$ be the countable subset of $\doi$ consisting of 
$\etta_{[a,b)}$ with $a,b\in\tNA_i$, together with $\etta_{\set1}$.
Then \as{} $h^i_j\in Q_i$, and thus if $D_1$ is the closed separable
subspace of $\doim$ generated by the countable set $\tensorim h_i$ with
$h_i\in Q_i$, then \as{} $g_n\in Q$ for every $n$, and thus \as{} $X\in Q$. 
\end{proof}

\section{Moments}\label{Smom}

For a random variable $X$ with values in some Banach space $B$, moments of
$X$ can be defined as $\E [X\tpx\ell]$, see \cite{SJ271}.
However, there are several possible interpretations of this; we may 
take the expectation in either the projective tensor power $B\ptpx\ell$ or
the injective tensor power
$B\itpx\ell$, and we can assume that the expectation exists in Dunford,
Pettis or Bochner sense, thus giving six different cases. 
See 
\cite{SJ271} 
(and the short summary in \refApp{Aintegrals})
for definitions and further
details; we recall here only the implications for existence:
\begin{gather*}
\text{projective $\implies$ injective},
\\
\text{Bochner $\implies$ Pettis $\implies$ Dunford},
\end{gather*}
and that if the moment exists in Bochner and Pettis sense, it is an element
of the tensor product, but a Dunford moment is in general
an element of the bidual of the tensor product.

In the special case $\ell=1$, when we consider the mean $\E[X]$, there is no
difference between the projective and injective case, but we can still
consider the mean in Bochner, Pettis or Dunford sense.

\subsection{Bochner and Pettis moments as functions}
We consider here the different moments when  $B=\doim$.
Recall first from \refSS{SStensor} that $\doim\itpl=\doix{\ell m}$. Hence,
if the injective moment
$\E [X\itpl]$ exists in Bochner or Pettis sense,
then this moment is an element of $\doix{\ell m}$, and thus a function on
$\oi^{\ell m}$. 

Moreover, recall also that $\doim$ has the approximation property and thus
the natural map $\doim\ptpl\to\doim\itpl=\doix{\ell m}$ is a continuous
injection. 
Hence, if the projective moment
$\E [X\ptpl]$ exists in Bochner or Pettis sense, then it 
too can be regarded as function in  $\doix{\ell m}$, and 
it equals the corresponding injective moment.
(Cf.\ \cite[Theorem 3.3]{SJ271}.)

It is easy to identify this function that is the moment (in any of these
four senses for which the moment exists).

\begin{theorem}\label{Tmoment}
Let $X$ be a \dmeas{} random variable in \doim{} 
and let $\ell\ge1$.
If $X$ has a projective or injective moment $\E [X\tpl]$
in Bochner or Pettis sense,
then this
moment $\E[X\tpx\ell]$
is the function in $\doix{\ell m}$
given by
\begin{equation}\label{moment}
  \E [X\tpx\ell](t_1,\dots,t_\ell)
=\E\lrpar{\prodil X(t_i)}, \qquad t_i\in \oi^m.
\end{equation}
\end{theorem}
In other words, the injective or projective 
Bochner or Pettis $\ell$-th moment (when it exists)
is the function describing all
mixed $\ell$-th moments of $X(t)$, $t\in\oi^m$.

\begin{proof}
As seen before the theorem, the moment can be regarded as a function in
\doix{\ell m}. 
  Since point evaluations are continuous linear functionals on
$\doix{\ell m}$, it follows that 
\begin{equation}\label{moment00}
  \E [X\tpx\ell](t_1,\dots,t_\ell)
=   \E [X\tpx\ell(t_1,\dots,t_\ell)]
=\E\lrpar{\prodil X(t_i)}, 
\end{equation}
showing \eqref{moment}.
\end{proof}

\begin{remark}\label{Rmoment}
\refT{Tmoment} does not hold for  Dunford moments, since a Dunford
  integral in general is an element of the bidual,
see Examples \ref{Edunford1}--\ref{Edunford2}.
However, we show a related result in \refT{Tdunford}, where we consider the
moment function \eqref{moment} for arguments in $\him$ and not just in $\oi^m$.
\end{remark}

\subsection{Existence of moments}
Since the Bochner and Pettis moments are given by \eqref{moment} when they
exist, the main problem is thus whether the different moments exist or not
for a given random random $X\in\doim$.
We give some conditions for existence, all generalizing
results in \cite{SJ271} for the case $m=1$.

For Bochner moments, we have a simple necessary and sufficient condition,
valid for both projective and injective moments.

\begin{theorem}\label{TBochner}
Let $X$ be a \dmeas{} \doim-valued \rv.
Then the following are equivalent.
\begin{romenumerate}
\item \label{TBochnerp}
The projective moment $\E X\ptpl$ exists in Bochner sense.
\item \label{TBochneri}
The injective moment $\E X\itpl$ exists in Bochner sense.
\item  \label{TBochner-0}
$\E \norm{X}^\ell<\infty$ and
there exist (non-random) countable subsets
$\NA_1,\dots,\NA_m$ of\/ $\oi$ 
such that for every $i\le m$,
\as{} $\gD_i X(t)=0$ for all $t=(t_1,\dots,t_m)$
with $t_i\notin \NA_i$.
\end{romenumerate}
\end{theorem}
\begin{proof}
By \refT{Tassep}, \ref{TBochner-0} is 
    equivalent to $\E\norm{X}^\ell<\infty$ and $X$ \assep.
The equivalence now follow
by \cite[Theorem 3.8]{SJ271}, 
since  $X$  is weakly measurable by \refC{CM3}.
\end{proof}

Unfortunately, the condition in \refT{TBochner}\ref{TBochner-0} 
shows that Bochner moments do not exist in
many applications, \cf{} \refR{Rassep}.
Hence the Pettis moments are more useful
for applications; the following theorem
gives a simple and widely applicable sufficient condition for their existence.
\begin{theorem}\label{TD1kp}
Let $X$ be a \dmeas{} \doim-valued \rv, and suppose that 
$\E \norm{X}^\ell<\infty$. 
Then the projective moment $\E X\ptpl$ 
and  injective moment $\E X\itpl$
exist in Pettis sense.
\end{theorem}
\begin{proof}
Recall that a bounded linear functional $\ga$ on $\doim\ptpl$ is the same as a
bounded $\ell$-linear form  $\ga:\doim^\ell\to\bbR$.
By \refT{TM}, $\innprod{\ga,X\tpl}=\ga(X,\dots,X)$ is \meas.
Furthermore,
\begin{equation}\label{gal}
\bigabs{  \innprod{\ga,X\tpl}}
\le\norm\ga \norm{X}^\ell,
\end{equation}
and it follows  that the family
$\bigset{\innprod{\ga,X\ptpl}:\norm\ga\le1}$ is \ui.
Moreover, $X\ptpl$ is  \wassep{} in $\doim\ptpl$ by \refT{TDwassep}.
Hence a theorem by Huff \cite{Huff}, 
see also \cite[Theorem 2.23 and Remark 2.24]{SJ271}, 
shows that $\E X\ptpl$ exists in Pettis sense. 

Since the natural inclusion $\doim\ptpl\to\doim\itpl$ is continuous,
the injective moment $\E X\itpl$ too exists in Pettis sense. 
\end{proof}

For injective moments, we can weaken the condition in \refT{TD1kp}, and
obtain a necessary and sufficient condition; there is also a corresponding
result for Dunford moments. 

\begin{theorem}\label{TD1}
Suppose that $X$ is a $\cD$-\meas{}  $\doim$-valued \rv, and
let $\ell\ge1$.
  \begin{romenumerate}[-10pt]
  \item \label{td1d}
$\E X\itpl$ exists in Dunford sense $\iff$ 
$\sup_{t\in \oi^m}\E|X(t)|^\ell<\infty$.
  \item \label{td1p}
$\E X\itpl$ exists in Pettis sense $\iff$ 
the family
$\set{|X(t)|^\ell:t\in \oi^m}$ of \rv{s} is \ui.
  \end{romenumerate}
\end{theorem}

We postpone the proof and show first two lemmas.

\begin{lemma}
  \label{LQ2}
Suppose that $X$ is a random element of $\doim$.
Then
\begin{equation}\label{ly}
  \sup_{\chi\in \doim^*,\;\norm{\chi}\le1}\E|\chi(X)| =
\sup_{t\in\oi^m}\E|X(t)|.
\end{equation}

\end{lemma}

\begin{proof}
Denote the left and right sides of \eqref{ly} by $L$ and $R$, and note that
trivially $R\le L$ because
each point evaluation $X\mapsto X(t)$ is a
linear functional of norm 1.

For the converse, 
let $\chi\in\doim^*$ with $\norm{\chi}\le1$.
Note that  the measurability of $\chi(X)$ follows from \refC{CM1}.
By the Riesz representation theorem (\refP{Priesz}), 
there exists a signed Baire measure
$\mu$ on $\hI^m$ with $\norm{\mu}=\norm{\chi}\le1$
such that $\chi(f)=\int_{\hI^m}f\dd\mu$ for every
$f\in\doim$.
Consequently, applying
\refT{Tfubini} to
 $|X(\xt)|$ and $|\mu|$,
noting that
\refL{Lfatou} yields
$  \E |X(\xt)|\le R$
for every $\xt\in\him$,
\begin{equation*}
  \begin{split}
  \E|\chi(X)|
&=
\E\biggabs{\int_{\hI^m} X(\xt)\dd\mu(\xt)}
\le \E\int_{\hI^m} |X(\xt)|\dd|\mu|(\xt)
\\&
=\int_{\hI^m} \E| X(\xt)|\dd|\mu|(\xt)
\le R\, |\mu|(\hI^m)= R\norm{\mu}\le R.
  \end{split}
\end{equation*}
Hence $L\le R$, which completes the proof.
\end{proof}

We extend this to powers.

\begin{lemma}\label{LP2}
Suppose that $X$ is a \doim-valued \rv, and $\ell\ge1$.
\begin{romenumerate}
\item \label{LP2E}
Then
\begin{equation}\label{lp2e}
  \sup_{\chi\in \doim^*,\;\norm{\chi}\le1}\E|\chi(X)|^\ell =
\sup_{t\in\oi^m}\E|X(t)|^\ell.
\end{equation}
In particular,
the set
$  \bigset{|\chi(X)|^\ell:\chi\in \doim^*,\;\norm{\chi}\le1}$
of random variables
is a bounded subset of $L^1$ if and only if 
the set
$  \bigset{|X(t)|^\ell: t\in\oi^m}$ is.
\item \label{LP2ui}
The set
$  \bigset{|\chi(X)|^\ell:\chi\in \doim^*,\;\norm{\chi}\le1}$
is \ui{} if and only if 
the set
$  \bigset{|X(t)|^\ell: t\in\oi^m}$ is.
\end{romenumerate}
\end{lemma}

\begin{proof}
\pfitemref{LP2E}
  Let $\chi\in \doim^*$ with $\norm{\chi}\le1$. 
Thus $\chi:\doim\to\bbR$ is a linear map, and we can take its tensor power
$\chi\tpl:\doim\itpl\to\bbR\itpl=\bbR$,
which is defined by
\begin{equation}
 \chi\tpl(f_1\tensor\dots\tensor f_\ell)=\prodil \chi(f_i) 
\end{equation}
together with linearity and continuity;
 $\chi\tpl$ is a linear functional on $\doim\itpl$ with norm
$\norm{\chi\tpl}=\norm{\chi}^\ell\le1$.

Recalling that $\doim\itpl=\doix{\ell m}$, we apply \refL{LQ2} to
$X\tpl\in\doix{\ell m}$ and the linear functional $\chi\tpl$ and obtain
\begin{equation}\label{kk}
\E|\chi(X)|^\ell=
  \E|\chi\tpl(X\tpl)|\le \sup_{t\in\oi^{\ell m}}\E |X\tpl(t)|. 
\end{equation}
Furthermore, if $t\in\oi^{\ell m}$, write $t=(t_1,\dots,t_\ell)$ with
$t_i\in\oi^m$; then by \Holder's inequality
\begin{equation}\label{kl}
  \E|X\tpl(t)|
=\E|X(t_1)\dotsm X(t_\ell)|
\le \prodil \bigpar{\E|X(t_i)|^\ell}^{1/\ell}
\le\sup_{t\in\oi^m}\E|X(t)|^\ell.
\end{equation}
Combining \eqref{kk} and \eqref{kl}, we see that the \lhs{} of \eqref{lp2e}
is at most equal to the \rhs.
The converse follows again because each point evaluation $X\mapsto X(t)$ is a
linear functional of norm 1.

\pfitemref{LP2ui}
Let $E\in\cF$ be an arbitrary event in the \ps{} $\ofp$, and apply
\eqref{lp2e} to
the random function $\etta_EX$. This yields
\begin{equation}
  \sup_{\chi\in \doim^*,\;\norm{\chi}\le1}\E\bigpar{\etta_E|\chi(X)|^\ell} 
= \sup_{t\in\oi^m}\E\bigpar{\etta_E|X(t)|^\ell}.
\end{equation}
The result follows, since a collection $\set{\xi_\ga}$
of random variables is \ui{} if and
only if it is bounded in $L^1$ and 
$\sup_{\P(E)<\gd}\sup_\ga\E|\etta_E\xi_\ga|\to0$ as $\gd\to0$
\cite[Theorem 5.4.1]{Gut}.
\end{proof}

\begin{proof}[Proof of \refT{TD1}]
\pfitemref{td1d}
$X$ is weakly measurable by \refC{CM3}, and \wassep{} by \refT{TDwassep}.  
Hence \cite[Theorem 3.11]{SJ271} shows that the injective Dunford moment
exists if and only if 
$  \bigset{|\chi(X)|^\ell:\chi\in \doim^*,\;\norm{\chi}\le1}$ is a bounded
subset of $L^1$.
The proof is completed by \refL{LP2}\ref{LP2E}.
\pfitemref{td1p}
Similar, using \cite[Theorem 3.20]{SJ271} and \refL{LP2}\ref{LP2ui}.
\end{proof}

For $\ell=1$, there is as said above no difference between projective and
injective moments.
For $\ell=2$, the projective and injective moments are expectations taken in
different spaces; nevertheless, the  projective moments exist if and only if
the injective moments do. For Bochner moments, this was shown in
\refT{TBochner} (for any $\ell$); for Pettis and Dunford moments this is
shown in the next theorem. This theorem does not hold for $\ell\ge3$; see
the counterexample in \cite[Example 7.27]{SJ271} (which is defined in $C(K)$
for another compact space $K$, but can be embedded in $\coi\subset\doi$).

\begin{theorem}\label{TDG}
Let $X$ be a \dmeas{} \doim-valued \rv.
  \begin{romenumerate}[-10pt]
  \item \label{tdgd}
$\E X\ptpx2$ exists in Dunford sense 
$\iff$
$\E X\itpx2$ exists in Dunford sense 
$\iff$
$\sup_{t\in	\oi^m}\E|X(t)|^2<\infty$.
  \item \label{tdgp}
$\E X\ptpx2$ exists in Pettis sense $\iff$ 
$\E X\itpx2$ exists in Pettis sense $\iff$ 
the family
$\set{|X(t)|^2:t\in \oi^m}$ of \rv{s} is \ui.
  \end{romenumerate}
\end{theorem}

\begin{proof}
The second equivalences in \ref{tdgd} and \ref{tdgp} are the case $\ell=2$
of \refT{TD1}.
Furthermore, the existence of a projective moment always implies the
existence of the corresponding injective moment.
Hence it suffices to show that in both parts, the final condition implies
the existence of the projective moment.

\pfitemref{tdgd}
Let $\ga$ be a bounded bilinear form on $\doim=C(\hI^m)$. By 
Grothen\-dieck's theorem \cite{Grothendieck-resume}, 
$\ga$ extends
to a bounded bilinear form on $L^2(\hI^m,\nu)$ for some Baire 
probability measure
$\nu$ on $\hI^m$; furthermore,
see \eg{} \cite[Theorem 7.20]{SJ271}, 
\begin{equation}\label{pingst}
|\ga(f,f)|
\le \kkg\norm{\ga}\int_{\hI^m}|f(t)|^2 \dd\nu(t),
\qquad f\in C(\hI^m),
\end{equation}
where $\kkg$ is a universal constant.
(In this version, $\kkg$ is at most
$2$ times the usual Grothendieck's constant, see \cite[Remark
7.21]{SJ271}.) 

Furthermore, applying \refL{LQ2}
to the random function $|X(t)|^2\in\doim$ yields 
\begin{equation}\label{gk}
\E\int_{\hI^m}|X(t)|^2\dd\nu \le \sup_{t\in\oi^m}\E|X(t)|^2.
\end{equation}
Consequently, by combining \eqref{pingst} and \eqref{gk},
\begin{equation}\label{tref}
\E|\ga(X,X)|
\le \kkg \norm\ga 
\E\int_{\hI^m}|X(t)|^2\dd\nu
 \le \kkg \norm\ga 
\sup_{t\in\oi^m}\E|X(t)|^2.
\end{equation}

If $\sup_{t\in\oi^m}\E|X(t)|^2<\infty$, then \eqref{tref} shows that
$\E|\ga(X,X)|<\infty$ for every bounded bilinear form $\ga$ on $\doim$,
which implies the existence of the projective Dunford moment $\E X\ptpx2$ by
\cite[Theorem 3.16]{SJ271}.

\pfitemref{tdgp}
Since the bounded linear functionals on $\doim\ptpx2$ are identified with
the bounded bilinear forms on $\doim$, 
\eqref{tref} means, equivalently, that for every $\gYa\in(\doim\ptpx2)^*$,
\begin{equation}\label{tref2}
\E|\gYa(X\ptpx2)|
\le \kkg \norm\gYa 
\sup_{t\in\oi^m}\E|X(t)|^2.
\end{equation}

Assume that the family
$\set{|X(t)|^2:t\in \oi^m}$ is \ui.
By applying \eqref{tref2} to $\etta_E X$ 
as in the proof of \refL{LP2}\ref{LP2ui},
we obtain 
that the family
$\bigset{\ga(X\ptpx2):\ga\in (\doim\ptpx2)^*,\,\norm\ga\le1}$ is \ui.
Moreover, $X\ptpx2$ is  \wassep{} by \refT{TDwassep}.
Hence $\E X\ptpx2$ exists in Pettis sense by Huff's theorem \cite{Huff}, 
cf.\ the proof of \refT{TD1kp}.
\end{proof}

\subsection{Dunford moments}

As said above, a Dunford moment in general is an element of the bidual of
the space, and thus \refT{Tmoment} does not hold for Dunford moments.
Examples \ref{Edunford1}--\ref{Edunford2} below illustrate this.
However, although even for $\ell=1$, the bidual $\doim^{**}$ is large and
unwieldy, it turns out that Dunford moments are always rather simple
elements of it, and that they have a representation as functions
generalising \refT{Tmoment}.
This extends to injective moments of any order $\ell$.

\begin{theorem}\label{Tdunford}
  Let $X$ be a \dmeas{} random variable in \doim{} 
and let $\ell\ge1$.
If $X$ has an injective Dunford moment $\E [X\itpl]$
then this
moment $\E[X\itpx\ell]$
is represented by the bounded Baire measurable function 
$\zeta$ on $\hI^{\ell m}$ defined by
\begin{equation}\label{Dmoment}
\zeta(\xt_1,\dots,\xt_\ell)
:=\E\lrpar{\prodil X(\xt_i)}, \qquad \xt_i\in \hI^m,
\end{equation}
in the sense that if $\chi$ is any continuous linear functional on
$\doim\itpl=\doix{\ell m}=C(\hilm)$ and $\chi$ is represented by a signed Baire
measure $\mu$ on $\hilm$, then
\begin{equation}\label{Dmoment2}
\innprod{\E[X\itpl],\chi}=\int_{\hilm} \zeta \dd\mu.
\end{equation}
In particular, if\/ $\zeta\in C(\hilm)$, then the Dunford moment\/ $\E[X\itpl]$
is the element $\zeta\in C(\hilm)=\doix{\ell m}$.
\end{theorem}

\begin{proof}
  Again, by considering the random variable 
$X\itpl\in\doim\itpl=\doix{\ell m}$, it suffices to consider the case
$\ell=1$.

In the case $\ell=1$,  the assumption says that $\E[X]$ exists
as a Dunford moment; by \refT{TD1}, this implies that
$\sup_{t\in\oi^m}\E|X(t)|<\infty$.
It follows from \refT{Tfubini}\ref{TfubiniB} that $\zeta(\xt):=\E [X(\xt)]$
is a bounded Baire measurable function on $\hI^m$, and that for any
continuous linear functional $\chi\in\doim^*$,
represented by a signed Baire measure $\mu\in\MBa(\him)$,
\begin{equation}
\E \innprod{\chi,X}
= \E \inthm X(\xt)\dd\mu(\xt)
=
\inthm \E[X(\xt)]\dd\mu(\xt)
=\inthm \zeta(\xt) \dd\mu(\xt).  
\end{equation}
This shows, cf.\ \refD{Ddunford}, that the Dunford integral is given by
$\innprod{\E[X],\chi}=\inthm \zeta(\xt) \dd\mu(\xt)$.
\end{proof}

Thus, similarly to the
Bochner or Pettis moments in \refT{Tmoment},
an injective Dunford  $\ell$-th moment is represented by
the function describing all
mixed $\ell$-th moments of $X(t)$, $t\in\hI^m$.
However, unlike the situation for Bochner and Pettis moments,
for Dunford moments we have,
in general, 
to consider $\zeta$ as a function of $\hI^{\ell m}$ and not just on
$\oi^{\ell m}$, see \refE{Edunford2} below.

\begin{remark}\label{RpiDunford}
For projective Dunford moments $\E[X\ptpl]$, the situation seems more
complicated. 
We have a continuous inclusion $i:\doim\ptpl\subseteq\doim\itpl$,
which induces a continuous linear map between the biduals
$i^{**}:(\doim\ptpl)^{**}\subseteq(\doim\itpl)^{**}$.
Thus,
if a projective Dunford moment $\E[X\ptpl]$ exists, then so does the
injective Dunford moment $\E[X\itpl]$, 
and can by \refT{Tdunford} be
represented by  the function $\zeta$ in \eqref{Dmoment}.
However, for $\ell\ge2$,
we do not know whether the map $i^{**}$ is injective
so that also the projective moment is represented  by $\zeta$. 
\end{remark}

We give two simple examples of Dunford moments, showing some bad
behaviour that may occur.
We take 
$m=1$ and $\ell=1$, \ie, we consider the mean $\E X$ of random variables $X$
in $\doi$.
\begin{example}
  \label{Edunford1}
Let 
$X=2^n\etta_{[2^{-n-1},2^{-n})}$ with probability $2^{-n}$, $n\ge1$.
Then $\E|X(t)|=\E X(t)=\etta_{(0,1/2)}(t)\le1$, 
for every $t\in\oi$,
and thus $\E X$ exists in Dunford
sense by \refT{TD1}\ref{td1d}. However, the function 
$\zeta(t):=\E [X(t)]$ is not
right-continuous at 0, so it does not belong to $\doi$; hence this function
does not represent $\E X$ in the sense of \refT{Tmoment}. In fact, it
follows that $\E X\in\doi^{**}\setminus\doi$. 

Nevertheless, \refT{Tdunford} shows that $\E X$ is represented by $\zeta$,
regarded as a function on $\hI$. It is easily seen that 
$\zeta(\xt):=\E [X(\xt)]=\etta_{(0,1/2)}(\xt)$ for all $\xt\in\hI$, and thus 
the Dunford mean $\E X$ is given by this function $\etta_{(0,1/2)}$ on
$\hI$; this function is bounded and \Bameas{} (as guaranteed by
\refT{Tdunford}), but  it is not continuous, and thus does not correspond to
an element of $\doi$. 

By \refT{TD1}, or \refT{Tmoment}, $\E X$ does not exist in Pettis (or
Bochner) sense.
\end{example}

\begin{example}
  \label{Edunford2}
 Let
$X=2^n\etta_{[1-2^{-n},1-2^{-n-1})}$ with probability $2^{-n}$, $n\ge1$.
Then $\E|X(t)|=\E X(t)=\etta_{[1/2,1)}(t)\le1$
for every $t\in\oi$,
and thus $\E X$ exists in Dunford
sense by \refT{TD1}\ref{td1d}. 
In this case, the function  $\zeta(t):=\E [X(t)]=\etta_{[1/2,1)}(t)$ 
is a function in $\doi$.
Nevertheless, the Dunford moment $\E X\in \doi^{**}$ cannot be identified
with the function $\zeta=\etta_{[1/2,1)}\in\doi$. 

To see this, we consider $\xt\in\hI$, as prescribed by \refT{Tdunford}.
 We have $X(1-)=0$ a.s., and thus \eqref{Dmoment} yields
$\zeta(1-):=\E[X(1-)]=0$. 
Hence, if $t_n\upto1-$, with $t_n\in(1/2,1)$, then
$\zeta(t_n)=1$ does not converge to $\zeta(1-)=0$, and thus $\zeta$ is not
continuous on $\hI$
at $1-$. 

Consequently, we see that also in this example, 
$\E X\in\doi^{**}\setminus\doi$. 
Nevertheless, \refT{Tdunford} shows that $\E X$ is represented by the
function $\zeta(\xt)=\etta_{[1/2,{1-})}(\xt)$ on $\hI$, 
and that \eqref{Dmoment2} holds.

By \refT{TD1}, $\E X$ does not exist in Pettis (or
Bochner) sense.

This example shows that it is necessary to consider the function $\zeta$
given by \eqref{Dmoment} as defined on $\hI^{\ell m}$, and not just on
$\oi^{\ell m}$. In the present example, $\zeta\notin\chI$, but its
restriction to $\oi$ is an element of $\doi$, and thus the restriction of
another function $\zeta'\in\chI$. \refT{Tdunford} shows that the mean $\E X$
is represented by $\zeta$, which is interpreted as an element of
$\chI^{**}\setminus\chI$ by \eqref{Dmoment2}, and not by $\zeta'\in\chI$.
\end{example}

If a Pettis moment exists, then the corresponding Dunford moment exists and
is equal to the Pettis moment. Theorems \ref{Tmoment} and \ref{Tdunford}
then yield two versions of the same representation; 
obviously \eqref{moment} is the restriction of \eqref{Dmoment} to $\oi^{\ell m}$;
we state a simple result showing the consistency of the extensions to $\hilm$.

\begin{theorem}
  \label{TDP0}
  Let $X$ be a \dmeas{} random variable in \doim{} 
and let\/ $\ell\ge1$.
If\/ $X$ has an injective Pettis moment\/ $\E [X\itpl]$,
then the function $\zeta$ in \eqref{Dmoment} is continuous on $\hI^{\ell  m}$,
and this element of $\chIlm=\doix{\ell m}=\doim\itpl$ equals the moment
$\E [X\itpl]$. 
\end{theorem}
\begin{proof}
The Pettis moment $\E[X\itpl]\in\doix{\ell m}=\chIlm$, and this function on
$\hI^{\ell m}$ equals $\zeta$ in \eqref{Dmoment}
by the calculation \eqref{moment00} extended to
$\xt_1,\dots,\xt_\ell\in\hI^m$, see also \cite[Theorem 7.10]{SJ271}.  
\end{proof}

There exists no general converse to this; even if the function $\zeta$ in
\eqref{Dmoment} is continuous on $\hilm$, the Pettis moment $\E[X\itpl]$
does not have to exist, as shown by the trivial  \refE{EP} below.
However, \refT{TP} shows that the implication holds in some cases.

\begin{example}
  \label{EP}
Take again $\ell=m=1$.
Let $Y=\eta X$, where $X$ is as in \refE{Edunford1} and $\eta=\pm1$, with
$\P(\eta=1)=\P(\eta=-1)=1/2$, with $X$ and $\eta$ independent.
Then $\zeta(\xt):=\E[X(\xt)]=0$ for every $\xt\in\hI$, so $\zeta\in C(\hI)$;
nevertheless $\E Y$ does not exist in Pettis sense by \refT{TD1} (or by the
definition \eqref{Dpettis}, taking $E:=\set{\eta=1}$).
\end{example}

\begin{theorem}\label{TP}
  Let $X$ be a \dmeas{} random variable in \doim{} 
and let $\ell\ge1$.
Suppose further that either
\begin{alphenumerate}
   \item $X(t)\ge0$ \as{} for every $t\in\oi$, or
  \item $\ell$ is even.
\end{alphenumerate}
Then the following are equivalent.
\begin{romenumerate}
\item \label{TPa}
$X$ has an injective Pettis moment $\E [X\itpl]$.
\item \label{TPb}
The function $\zeta$ in \eqref{Dmoment} 
exists everywhere and is continuous on $\hI^{\ell  m}$.
\item \label{TPc}
$\E |X(\xt)|^\ell<\infty$ for every $\xt\in\him$, 
and the function $\xt\mapsto\zeta(\xt,\dots,\xt):=\E[X(\xt)^\ell]$ 
is continuous on $\him$.
 \end{romenumerate} 
\end{theorem}
\begin{proof}
  \ref{TPa}$\implies$\ref{TPb}:
By \refT{TDP0}.

  \ref{TPb}$\implies$\ref{TPc}:
Trivial.

  \ref{TPc}$\implies$\ref{TPa}:
In both cases we have $X(\xt)^\ell\ge0$ and thus $|X(\xt)|^\ell=X(\xt)^\ell$ a.s.
The argument in the proof of the similar \cite[Theorem 7.19]{SJ271} shows
that
the family $\set{|X(\xt)|^\ell:\xt\in\him}$ of random variables is
uniformly integrable.
This proof in \cite{SJ271} is stated for $C(K)$ when $K$ is a metrizable
compact, but in the part of the proof used here, metrizability is used only
to show that a sequentially continuous function on $K$ is continuous, and
this holds for every first countable compact $K$
\cite[Theorem 1.6.14 and Proposition 1.6.15]{Engelking}, and thus for
$\hI^m$.

The result \ref{TPa} now follows from \refT{TD1}.
\end{proof}

\section{An application to Zolotarev distances}

\subsection{Equal moments}
As a corollary of the results on moment above, we obtain the following
results on equality of moments of two different \doim-valued \rv{s}.

\begin{theorem}  \label{TXY1}
Let $X$ and $Y$ be \dmeas{} \doim-valued \rv{s}, and let $\ell\ge1$.
Suppose that 
$\E \norm{X}^\ell, \E \norm{Y}^\ell<\infty$. 
Then  the moments in \ref{TXY1ptp} and \ref{TXY1itp} below exist in 
Pettis sense, and
the following are equivalent.
\begin{romenumerate}

\item  \label{TXY1oi} 
For every $t_1,\dots,t_\ell\in\oi^m$,
\begin{equation}\label{txy1oi}
  \E\lrpar{\prodil X(t_i)}
=\E\lrpar{\prodil Y(t_i)}.
\end{equation}  

\item   \label{TXY1ga}
For every continuous $\ell$-linear form $\ga$ on $\doim$,
  \begin{equation}\label{txy1ga}
\E \ga(X,\dots,X)=\E \ga(Y,\dots,Y) .   
  \end{equation}

\item   \label{TXY1ptp}
$\E [X\ptpl]=\E [Y\ptpl]$. 

\item   \label{TXY1itp}
$\E [X\itpl]=\E [Y\itpl]$. 
\end{romenumerate}
\end{theorem}

Note that \ref{TXY1oi} is a special case of \ref{TXY1ga}; the converse
implication \ref{TXY1oi}$\implies$\ref{TXY1ga} is far from trivial and is
the main content of this theorem.

\begin{proof}
First, the Pettis moments in \ref{TXY1ptp} and \ref{TXY1itp} exist
by \refT{TD1kp}. The assumptions imply also that the expectations in
\eqref{txy1oi} and \eqref{txy1ga} are finite, see \eqref{gal}.

\ref{TXY1oi}$\iff$\ref{TXY1ptp}$\iff$\ref{TXY1itp}:
By \refT{Tmoment}.

\ref{TXY1ga}$\iff$\ref{TXY1ptp}:
The equality $\E [X\ptpl]=\E [Y\ptpl]$ holds if and only if
we have
$\innprod{\ga,\E [X\ptpl]}=\innprod{\ga,\E [Y\ptpl]}$
for every continuous linear functional $\ga$ on $\doim\ptpl$; 
these $\ga$ can be identified with the continuous $\ell$-linear forms on
$\doim$,  
and the result follows since
\begin{equation}
  \innprod{\ga,\E X\ptpl}
=\E\innprod{\ga, X\ptpl}
=\E\innprod{\ga(X,\dots,X)}
\end{equation}
and similarly for $Y$.
\end{proof}

 \subsection{Zolotarev distance} 
The \emph{Zolotarev distance} $\zeta_s(X,Y)$ between two 
random variables
$X$ and $Y$ with values in a Banach space, 
or more precisely between their distributions $\cL(X)$ and $\cL(Y)$,
with $s>0$ a real parameter,
was defined by \citet{Zolotarev76};
we refer to \cite{Zolotarev76} 
or to \cite[Appendix B]{SJ271} and the further
references there for the definition and basic properties.

The Zolotarev distance is a useful tool to show approximation and
convergence of distributions.
In order to apply the Zolotarev distance to a problem, the first step is to
show that the distance $\zeta_s(X,Y)$ between two given
random variables is finite. 
It was shown in \cite[Lemma B.2]{SJ271} that,
assuming that $X$ and $Y$ are weakly measurable and that 
$\E\norm{X}^s,\E\norm{Y}^s<\infty$,
this holds 
if and only if the projective (Dunford) moments
$\E X\ptpl$ and $\E Y\ptpl$ are equal for every positive integer $\ell<s$.
(This condition is vacuous if $0<s\le1$.)

For the case of random variables in $\doim$, this and the results above
yield the following simple criterion, which extends results for the case $m=1$ 
in \cite{SJ271}.

\begin{theorem}  \label{TZ}
Let $X$ and $Y$ be \dmeas{} \doim-valued \rv{s}, and let $s>0$.
Suppose that 
$\E \norm{X}^s, \E \norm{Y}^s<\infty$. 
Then the following are equivalent.
\begin{romenumerate}

\item \label{TZZ}
The Zolotarev distance $\zeta_s(X,Y)<\infty$.

\item  \label{TZoi} 
For every positive integer $\ell<s$ and every
$t_1,\dots,t_\ell\in\oi^m$,
\begin{equation}\label{tz1oi}
  \E\lrpar{\prodil X(t_i)}
=\E\lrpar{\prodil Y(t_i)}.
\end{equation}  


%
\end{romenumerate}
\end{theorem}

\begin{proof}
  By \refC{CM3}, $X$ and $Y$ are \wmeas, and thus \cite[Lemma B.2]{SJ271}
  applies and shows, as said above, that 
\ref{TZZ}$\iff \E [X\ptpl]=\E [Y\ptpl]$,
which is equivalent to \ref{TZoi} by \refT{TXY1}.
\end{proof}

\subsection{Further results and comments}

For injective moments,
we can weaken the moment assumption in
\refT{TXY1}.
\begin{theorem}  \label{TXY2}
Let $X$ and $Y$ be \dmeas{} \doim-valued \rv{s},
and let $\ell\ge1$.
Suppose further that 
$\sup_{t\in\oi^m}\E |X(t)|^\ell<\infty$ and $
\sup_{t\in\oi^m}\E |Y(t)|^\ell<\infty$.
Then  the injective moments in \ref{TXY2itp} below exist in Dunford
sense, and
the following are equivalent.
\begin{romenumerate}

\item  \label{TXY2hI} 
For every $\xt_1,\dots,\xt_\ell\in\hI^m$,
\begin{equation}\label{txy2hI}
  \E\lrpar{\prodil X(\xt_i)}
=\E\lrpar{\prodil Y(\xt_i)}.
\end{equation}

\item   \label{TXY2itp}
$\E [X\itpl]=\E [Y\itpl]$. 
\end{romenumerate}
\end{theorem}

\begin{proof}
The injective Dunford moments in  \ref{TXY2itp} exist
by \refT{TD1}.

The equivalence
\ref{TXY2hI} $\iff$\ref{TXY2itp}
follows by \refT{Tdunford}.
\end{proof}

Note that we consider arbitrary $\xt_i\in\hI$ in \eqref{txy2hI}, unlike in
\eqref{txy1oi}; 
this is necessary as is shown by the following example.

\begin{example}
  Take $\ell=m=1$.
Let $X$ be as in \refE{Edunford2}, and let $Y$ be the deterministic function
$\etta_{[1/2,1)}\in\doi$.
Then,
as shown in \refE{Edunford2}, 
$\E[X(t)]=\E[Y(t)]$ for every $t\in\oi$, so \eqref{txy1oi} holds,
but $\E[X(1-)]=0\neq 1=\E[Y(1-)]$, and \eqref{txy2hI} fails;
consequently,  neither \ref{TXY2hI} nor \ref{TXY2itp} in \refT{TXY2} holds.
\end{example}

For $\ell=1$, there is no difference between injective and projective
moments, and thus \refT{TXY2} applies to projective moments as well.

For $\ell=2$, \refT{TDG} shows that the assumptions of \refT{TXY2} imply
also existence of the projective Dunford moments 
$\E X\ptpx2$ and $\E Y\ptpx2$.
However, we do not know whether they always are equal
when the injective moments are,
see also \refR{RpiDunford}.
\begin{problem}
Assume that the assumptions of \refT{TXY2} hold with $\ell=2$.
Are \ref{TXY2hI} and \ref{TXY2itp} equivalent also to
$\E [X\ptpl]=\E [Y\ptpl]$?
\end{problem}

\appendix
\section{Baire and Borel sets in $\hI^m$}\label{Abaire}

We show here the claims in \refE{Ebaire}, and give some further results.
The results are presumably known, but we have not found a reference and give
proofs for completeness.

\subsection{The case $m=1$}

Recall from \refL{Lbaire}\ref{baire=} that 
the Baire and Borel \gsf{s} coincide
for every
metrizable compact space; in particular $\Ba(I)=\cB(I)$.
The space $\hI$ is compact but not metrizable; nevertheless, as shown below,
the Baire and Borel \gsf{} coincide there too.

Recall also that $\rho:\hI\to I$ is the natural projection.

\begin{proposition}\label{P1}
  \begin{equation}
  \Ba(\hI)
=
\cB(\hI)
=
\bigset{\rho\qw(A)\diff N: A\in\cB(I),\;
N\subset\hI \text{ with }|N|\le\aleph_0 }.
  \end{equation}
\end{proposition}

In other words, the Borel (or Baire) sets in $\hI$ are obtained from the
Borel sets in $I$ in the natural way (by identifying $t$ and $t-$), except
that there may be a countable number of $t$ such that the set contains only
one of $t$ and $t-$. 

\begin{proof}
Let $\cG:=\bigset{\rho\qw(A)\diff N: A\in\cB(I),\;
N\subset\hI \text{ with }|N|\le\aleph_0 }$;
$\cG$ is easily seen to be a \gsf.
We prove three inclusions separately.

(i) $\Ba(\hI)\subseteq\cB(\hI)$. Trivial.

(ii) $\cB(\hI)\subseteq\cG$.
Note first that an open interval in $\hI$ always is either of the form
$\rho\qw\bigpar{(a,b)}$ for some open interval $(a,b)\subset(0,1)$, or of
this form with one or two of the endpoints $a$ and $b-$ added; if $b=1$ we
may also add 1.
Let $\hU\subseteq\hI$ be open; then $\hU$ is a union of 
a (possibly uncountable) set of 
open intervals $
\hU_\ga\subseteq \hI$. For each $\hU_\ga$, let $V_\ga=(a_\ga,b_\ga)$ be the
corresponding open interval in $I$; 
thus $\hU_\ga\supseteq\rho\qw(V_\ga)$ and $\hU_\ga\setminus\rho\qw(V_\ga)$
consists of at most the two endpoints and 1.
Let $V:=\bigcup_\ga V_\ga$; this is an open subset of $(0,1)$.
Consequently, $V=\bigcup_j W_j$ for some countable collection of open
disjoint intervals $W_j=(c_j,d_j)\subset(0,1)$. 

Consider one of the intervals $V_\ga=(a_\ga,b_\ga)$. If $a_\ga\in V$, then
$a_\ga\in\rho\qw(V)$, and if $a_\ga\notin V$, then $a_\ga$ equals one of the
endpoints $c_j$. Similarly, either $b_\ga-\in\rho\qw(V)$ or $b_\ga$ equals
some endpoint $d_j$.
Consequently, $\hU=\bigcup_\ga\hU_\ga\supseteq \rho\qw(V)$, and 
$\hU\setminus\rho\qw(V)$ is a subset of the countable set
$\set{c_j,d_j-,1}$.
Consequently, $\hU\in\cG$.

This shows that $\cG$ contains every open subset of $\hI$, and thus the
Borel \gsf{} $\cB(\hI)$.

(iii) $\cG\subseteq\Ba(\hI)$:
By \refL{L1}\ref{L1rho}, the mapping $\rho$ is continuous and thus Baire
measurable; hence $\rho\qw(A)\in\Ba(\hI)$ for every $A\in\cB(I)=\Ba(I)$.

  If $\xt\in\hI$, then the singleton \set{\xt} 
is closed, and a $G_\gd$ set; hence \set{\xt} is a Baire set.
Consequently, every countable subset $N$ of $\hI$ is a Baire set.

It follows that $\cG\subseteq\Ba(\hI)$.
\end{proof}

\begin{corollary}\label{CP1}
  Every closed subset of $\hI$ is a $G_\gd$ and every open subset of $\hI$
  is an $F_\gs$.
\end{corollary}
\begin{proof}
 The two parts are obviously equivalent. If $F$ is a closed, and thus
 compact, subset, then $F$ is a Borel set and thus by the proposition a
 Baire set. By \cite[Theorem 51D]{Halmos}, every compact Baire set is a
 $G_\gd$. 
\end{proof}

Since the Baire \gsf{} is generated by the  compact $G_\gd$ sets, the
corollary is equivalent to the proposition.

\begin{corollary}\label{CP1b}
Every finite Borel measure on $\hI$ is regular.
\end{corollary}
\begin{proof}
  Every finite Baire measure is regular
\cite[Theorem 52G]{Halmos}.
\end{proof}

\subsection{The case $m\ge2$}

The equality of the Baire and Borel \gsf{s} in \refP{P1} does not extend to
$\hI^m$ for $m>1$. We begin with the case $m=2$.

\begin{proposition}\label{P2}
  $\cB(\hI)\times\cB(\hI)
=\Ba(\hI)\times\Ba(\hI)
=\Ba(\hI^2)
\subsetneq\cB(\hI^2)$.
\end{proposition}

\begin{proof}
The first equality follows by \refP{P1}, and the second by
\refL{Lbaire}\ref{baire-times}. The final inclusion is trivial, and it
remains to show that it is strict.

  Let $h:I\to\hI^2$ be given by $h(t)=(t,1-t)$.
Thus, if we write $h=(h_1,h_2)$, then $h_1$ is the inclusion $\iota$ in
\refSS{SSnotation}, and $h_2(t)=\iota(1-t)$. By \refL{L1}, $\iota$ is
measurable $(I,\cB)=(I,\Ba)\to(\hI,\Ba)$, and thus both $h_1$ and $h_2$ are
measurable $(I,\cB)\to(\hI,\Ba)$;
Hence, if $E\subset\Ba(\hI^2)=\Ba(\hI)\times\Ba(\hI)$, then $h\qw(E)$ is a
Borel set in $I$.

On the other hand, for any $t\in\oi$, $[t,1]$ and $[1-t,1]$ are open
intervals in $\hI$.
Now let $A\subseteq I$ be arbitrary, and define a subset of $\hI^2$ by
\begin{equation}
  E_A:=\bigcup_{t\in A} \bigpar{[t,1]\times[1-t,1]}.
\end{equation}
This is an open subset of $\hI^2$, and thus a Borel set, \ie,
$E_A\in\cB(\hI^2)$. 

However, $h\qw(E_A)=A$. Hence, if we take a set $A$ that is not a Borel set,
then $h\qw(E_A)$ is not a Borel set, and thus, by the first part of the
proof, $E_A\notin\Ba(\hI^2)$.
\end{proof}

It follows easily from \refP{P2} that $\Ba(\hI^m)\subsetneq\cB(\hI^m)$ for
every $m>2$ too; we omit the details.

Note also that 
\refP{P2} implies that there exists
a closed set in $\hI^2$ that is not $G_\gd$;
cf.\ \refC{CP1}.

\begin{remark}
  \refP{P2} gives one proof that $\hI$ is not metrizable, since if
it were, then $\hI^2$ would be as well, which would contradict
\refL{Lbaire}\ref{baire=}. 
(Another proof uses that the proof above shows that the set
$\set{(t,1-t):t\in I}$ is an uncountable discrete subset of $\hI^2$; this is
impossible for a compact metric space.)
\end{remark}

\section{Integration in Banach spaces}\label{Aintegrals}

Let $f$ be a function defined on a measure space $(S,\cS,\mu)$ with values
in a Banach space $B$. Then there are (at least) three different ways to
define the integral $\int_Sf\dd\mu$; 
the three definitions apply  to different classes
of functions $f$, but when two or all three definitions apply to a function
$f$, then the integrals coincide. 
We use all three integrals in \refS{Smom} in the
case when $(S,\cS,\mu)$ is a probability space and the integrals can be seen
as expectations.

We give here a brief summary, and refer to \cite{SJ271}
and the reference given there for further details.

\subsection{Bochner integral}

The Bochner integral is a straight-forward generalization of the Lebesgue
integral to Banach-space valued functions.

\begin{theorem}
  A function $f$ is Bochner integrable if and only if $f$ is \Bormeas, \assep,
  and $\int_S\norm{f}\dd\mu<\infty$.
\nopf
\end{theorem}

The Bochner integral $\int_S f\dd\mu$ then is an element of $B$.

Unfortunately, as discussed in \refS{Smom}, the condition of \assep{} makes
the Bochner integral unapplicable in many interesting examples of
$\doim$-valued random variables.

\subsection{Dunford integral}
The Dunford integral is the most general of our integrals.

\begin{definition}\label{Ddunford}
  A function $f:S\to B$ is Dunford integrable if 
$x\mapsto\innprod{\chi,f(x)}$ is
  integrable 
(and in particular measurable)
on $S$ for every continuous linear functional $\chi\in B^*$.
In this case, as a consequence of the closed graph theorem,
there exists a (unique) element $\int_S f\dd\mu\in B^{**}$ 
(the \emph{Dunford integral})
such that
\begin{equation}
\int_S \innprod{\chi,f(x)}\dd\mu=\innprod{\int_S f\dd\mu,\chi}, 
\qquad \chi\in B^*.  
\end{equation}
\end{definition}
Note that the Dunford element is defined as an element of the bidual
$B^{**}$; in general, it is not an element of $B$.
(See \refE{Edunford1}.)

\subsection{Pettis integral}

A  Pettis integral is a Dunford integral that has its value in $B$; 
furthermore, the following is required (in order to have useful properties).
Note that if $f$ is Dunford integrable over $S$, then $f$ is always Dunford
integrable over every measurable subset $E\subseteq S$.

\begin{definition}\label{Dpettis}
  A function $f:S\to B$ is Pettis integrable if $f$ is Dunford integrable
  with $\int_S f\dd\mu\in B$, and, moreover,
$\int_E f\dd\mu\in B$ for every measurable subset $E\subseteq S$.
\end{definition}

By definition, a Pettis integrable function is also Dunford integrable, and
the two integrals coincide.
Similarly, it is easy to see that a Bochner integrable function is Pettis
integrable  (and thus also Dunford integrable) and that the integrals coincide.
The converses do not hold, in general.


\newcommand\AAP{\emph{Adv. Appl. Probab.} }
\newcommand\JAP{\emph{J. Appl. Probab.} }
\newcommand\JAMS{\emph{J. \AMS} }
\newcommand\MAMS{\emph{Memoirs \AMS} }
\newcommand\PAMS{\emph{Proc. \AMS} }
\newcommand\TAMS{\emph{Trans. \AMS} }
\newcommand\AnnMS{\emph{Ann. Math. Statist.} }
\newcommand\AnnPr{\emph{Ann. Probab.} }
\newcommand\CPC{\emph{Combin. Probab. Comput.} }
\newcommand\JMAA{\emph{J. Math. Anal. Appl.} }
\newcommand\RSA{\emph{Random Struct. Alg.} }
\newcommand\ZW{\emph{Z. Wahrsch. Verw. Gebiete} }
\newcommand\DMTCS{\jour{Discr. Math. Theor. Comput. Sci.} }

\newcommand\AMS{Amer. Math. Soc.}
\newcommand\Springer{Springer-Verlag}
\newcommand\Wiley{Wiley}

\newcommand\vol{\textbf}
\newcommand\jour{\emph}
\newcommand\book{\emph}
\newcommand\inbook{\emph}
\def\no#1#2,{\unskip#2, no. #1,} 
\newcommand\toappear{\unskip, to appear}

\newcommand\arxiv[1]{\texttt{arXiv:#1}}
\newcommand\arXiv{\arxiv}

\def\nobibitem#1\par{}

\enlargethispage{\baselineskip}

\end{document}